\documentclass[12pt]{amsart}
\usepackage{amssymb,latexsym}
\usepackage[all]{xy}
\newtheorem{thm}{Theorem}[section]
\newtheorem*{thm*}{Theorem}
\newtheorem{proposition}[thm]{Proposition}
\newtheorem{lemma}[thm]{Lemma}
\newtheorem{cor}[thm]{Corollary}
\theoremstyle{definition}

\newtheorem{remark}[thm]{Remark}
\newtheorem{claim}[thm]{Claim}

\renewcommand{\det}{{\rm det}\,}
\renewcommand{\deg}{{\rm deg}\,}
\newcommand{\by}[1]{\stackrel{#1}{\longrightarrow}}

\newcommand{\rank}{{\rm rank}\,}

\newcommand{\sEnd}{\mbox{${\sE}{nd}$}}
\newcommand{\sHom}{\mbox{${\sH}{om}$}}

\newcommand{\coker}{{\rm coker}\,}
\newcommand{\boxtensor}{{\Box\kern-9.03pt\raise1.42pt\hbox{$\times$}}}

\newcommand{\tensor}{\otimes}

\newcommand{\bP}{{\mathbf P}}
\newcommand{\sA}{{\mathcal A}}
\newcommand{\sC}{{\mathcal C}}
\newcommand{\sB}{{\mathcal B}}
\newcommand{\sE}{{\mathcal E}}

\newcommand{\sF}{{\mathcal F}}

\newcommand{\sH}{{\mathcal H}}

\newcommand{\sL}{{\mathcal L}}
\newcommand{\sM}{{\mathcal M}}

\newcommand{\sO}{{\mathcal O}}

\newcommand{\sT}{{\mathcal T}}  
\newcommand{\sU}{{\mathcal U}}
\newcommand{\sV}{{\mathcal V}}
\newcommand{\sW}{{\mathcal W}} 

\renewcommand{\tilde}{\widetilde}
\renewcommand{\hat}{\widehat}
\numberwithin{equation}{section}
\newcounter{elno}                % This to number lists

\newcounter{example}[section]
\def\theexample{\thesection.\arabic{example}}

\begin{document}

\title{semistablity of syzygy bundles on projective spaces in positive
characteristics}
{}

\author{V. Trivedi}
\address{School of Mathematics, Tata Institute of
Fundamental Research,
Homi Bhabha Road, Mumbai-400005, India}
\email{vija@math.tifr.res.in}

\thanks{preliminary version}
\date{}
\maketitle
\section{Introduction}
 Let  $k$ be an algebraically closed field.
 For an integer  $d>0$, let $\sV_d$ be the 
vector bundle on
$\bP_k^n$ given by the exact sequence
\begin{equation}\label{tt} 0 \by{} {\sV_d} \by{} H^0(\mathbf{P}_k^n,
\sO_{{\mathbf
P}^n_k}(d))\tensor
\sO_{{\mathbf P}^n_k} \by{\eta} \sO_{{\mathbf P}^n_k}(d) \by{}
0,\end{equation}
where $\eta$ is the evaluation map.

It was proved by Flenner [F] that if characteristic$~k = 0$ then $\sV_d$ 
is
a semistable vector bundle. 
 He uses this as an  crucial ingredient to prove his restriction 
theorem 
for 
torsion free  semistable sheaves on a normal projective variety, defined 
over a 
field of characteristic $0$, to a general hypersurface of degree $d$, 
where $d$ has 
a
lower bound in terms of degree of the ambient variety and degree 
and 
rank of the sheaf. He reduces the 
argument  to projective space and then uses the semistability property of 
$\sV_d$.

In characteristic~$k = p > 0$,  A. Langer ([L], proved the following 
restriction theorem for strongly semistablilty:

\vspace{5pt}

\begin{thm}\label{t1}(Langer)~~Let $(X, H)$ be a 
smooth 
$n$-dimensional $(n\geq 2)$ polarized 
variety with globally generated tangent bundle $\sT_X$. Let $E$ be a
$H$-semistable torsion free sheaf of rank $r\geq 2$ on $X$. Let 
 $d$ be an integer such that 
$$d > \frac{r-1}{r}\Delta(E)H^{n-2}+\frac{1}{r(r-1)H^n}$$
and 
$$\frac{{\binom{d+n}{d}}-1}{d} > H^n{\mbox{max}}\{\frac{r^2-1}{4}, 
1\}+1.$$
If $\mbox{characteristic}~k > d$ then the restriction $E_D$ is strongly 
$H$-semistable 
for a very general $D\in |dH|$.\end{thm}

\vspace{5pt}

However,  he has to assume that  
$\mbox{characteristic}~k = p > d$;
as he uses the proof of the result of Flenner, and more specifically the 
semistability property of $\sV_d$. In particular, for a given $p > 0$ his 
result is valid for 
at most finitely many $d$, in fact the set of such $d$ can be empty. In 
the end of the proof of  
Theorem~\ref{t1}, Langer remarked that the assumption on the 
characteristic 
could be 
removed if there is a positve answer to one of the following questions:

{\it Is $\sV_d$ a semistable bundle, 
for arbitrary $n$, $d$, and $p = \mbox{char}~k$?, or is there a good 
estimate on $\mu_{max}(\sV_d^{\star})$?}

 We recall that 
if $\mbox{char}~k = p > d$ or $\mbox{char}~k = 0$, then 
 $\sV_d$ is filtered by $S^m(\sV_1)\tensor \sO(d-m)$, and these are the 
only 
possible subquotients of $\sV_d$ as homogeneous bundles.
 However, as soon as $d$ exceeds 
$p$, many more subquotients of $\sV_d$ occur, and therefore 
argument of [F] is not applicable.

In this paper, we prove semistability of the syzygy bundle $\sV_d$, 
where $\mbox{char}~k = p>0$, under the 
conditions as given in Theorems~\ref{c1}. This 
 provides evidence in favour of a
positive solution in general.

\begin{thm}\label{c1}The vector bundle
 $\sV_d$,  given by the short exact sequence~(\ref{tt}), is
semistable (in fact stable) in any of the following cases:
\begin{enumerate} 
\item $\bP^n_k = \bP^2_k$ and $d\geq 1$, or
\item $d= a_ip^i+a_mp^m$ is the $p$-adic expansion of the integer 
$d$, for 
any $m\geq 0$, or
\item $d = (a_{0}+ a_{1}p + \cdots + a_{m}p^{m})p^{i_0}$
is
the $p$-adic expansion of
the integer $d$, where  $a_0$ and $a_m$ are nonzero integers,
such that one of the following holds,
\begin{enumerate}
\item $p \leq n$ and $a_{2},\ldots, a_{m} \geq 1$, or
\item $p \geq n$ and $a_{2}, \ldots, a_{m} \geq p-n+1$, or
\item $n\geq (a_0+a_1p+\cdots +a_mp^m)/p$.
\end{enumerate}

\end{enumerate}
\end{thm}

By  analysing further the proof of Theorem~\ref{c1}, 
we answer the second question of Langer affirmatively in the following 

\begin{proposition}\label{p2}Let $\sV_d$ be the vector bundle on 
$\bP^n_k$, given by 
the short exact sequence (\ref{tt}). Let $\sV_d^* = 
{\sHom}_{\sO_{\bP^n_k}}(\sV_d,\sO_{\bP^n_k})$ denote the dual of $\sV_d$. 
Then 
$$ \frac{d}{{\binom{d+n}{d}}-1}\leq \mu_{max}(\sV_d^*)\leq 
\frac{d}{\binom{ {\lceil{d/2}\rceil}+n-1}{{\lceil{d/2}\rceil}}}
,$$
where ${\lceil{x}\rceil} = $ the smallest integer $\geq x$.
\end{proposition}

As a consequnce one can remove the  
restriction on the characteristic of the field in Theorem~\ref{t1} of 
Langer and obtain the following 

\begin{cor}\label{c4} {\it Let $(X, H)$ be a smooth 
$n$-dimensional $(n\geq 2)$ polarized 
variety with globally generated tangent bundle $\sT_X$. Let $E$ be an 
$H$-semistable torsion free sheaf of rank $r\geq 2$ on $X$. Let 
 $d$ be an integer such that 
$$d > \frac{r-1}{r}\Delta(E)H^{n-2}+\frac{1}{r(r-1)H^n}$$
and, 
\begin{enumerate}
\item for  $n=2$,  
$$\frac{{\binom{d+n}{d}}-1}{d} > H^n{\mbox{max}}\{\frac{r^2-1}{4}, 
1\}+1,$$
\item and, for $n\geq 3$,  
$$\frac{\binom{ {\lceil{d/2}\rceil}+n-1}{{\lceil{d/2}\rceil}}}{d} > 
H^n{\mbox{max}}\{\frac{r^2-1}{4}, 
1\}+1.$$
\end{enumerate}
Then the restriction $E_D$ is strongly 
$H$-semistable 
for a very general $D\in |dH|$.}\end{cor}

It follows from the above corollary that, for a given char~$k = p > 0$, 
one can  find $d_0$ such that, for all $d \geq d_0$, the restriction $E_D$ 
is strongly $H$-semistable for a very general $D\in |dH|$.

We recall that (1) when $X$ is a smooth projective variety with  a 
polarization $H$ and $E$ a
strongly $H$-semistable bundle  on $X$ of rank $< \dim~X$ then 
Maruyama [Ma] 
proved 
that $E|_D$ is strongly $H$-semistable for a very general $D\in |H|$, 
{\it 
i.e.}, $d = 1$ in this case, (2) when $E$ is a homogeneous bundle on 
${\mathbf P}^n_k$ induced by an irreducible representation of $P$, where 
$P$ is a maximal parabolic subgroup of $GL(n+1)$ such that $GL(n+1)/P = 
{\mathbf P}^n_k$, then  $E|_D$ is  strongly semistable,  (i) if $D$ is  
any smooth quadric and $\mbox{char}~k \neq 2$ and (ii) if $D$ is any 
smooth 
cubic and $\mbox{char}~k \neq 3$. 
In particular $E|_D$ is strongly semistable for very general hypersurface 
of degree $\geq 2$, if $\mbox{char}~k > 5$.

\begin{remark}\quad One can prove the semistability of $\sV_d$ in the 
following 
case also, but the proof gets very technical (see arXiv:mathRT/0804.0547): 
Let $\sV_d$ be  the 
vector bundle as given by the short exact sequence~(\ref{tt}). 
Suppose  $d = (a_{0}+ a_{1}p + \cdots + a_{m}p^{m})p^{i_0}$
is the $p$-adic expansion of
the integer $d$, where  $a_0$ and $a_m$ are nonzero integers,
such that 
\begin{enumerate}
\item $n\geq m+1$,
\item $h^0(\bP^n_k, \sO_{\bP^n_k}(a_m))\geq 1+a_mmn$, and 
\item $a_0\leq a_1\leq \cdots <a_{m-1}$ and $a_{m-2} \leq a_m$.
\end{enumerate}
Then $\sV_d$ is a semistable (in fact stable) vector bundle over 
$\bP^n_k$.
 
If $a_m \geq 4$ then the condition~(2) in the above statement is always 
satisfied. Moreover if $m\leq n-2$ then the condition~(2) is satisfied for 
any $a_m \geq 3$ \end{remark}

Author would like to thank V.B. Mehta for suggesting the 
semistability question of  $\sV_d$.

\section{Syzygy bundles on $\bP^2_k$}
First we prove the following result for $\bP^2_k$, by different methods 
than
we use
for
higher dimensional projective space. 

\begin{proposition}\label{p1}For $n =2$ and for $d\geq 1$, the
bundle $\sV_{d}$  is  stable, 
on ${\mathbf P}^2_k$.
\end{proposition}

The proof relies on the following lemma, which we prove using an argument
similar to the proof the following proposition in [KR].

\vspace{5pt}

\noindent{\bf{Proposition}}~{\mbox [KR]}\quad {\it Let $X$ be a
nonsingular
curve
of genus
$g\geq 2$. Then for the pair $(X,\omega_X)$, where $\omega_X$ is the
 canonical line bundle of $X$, the sheaf $K_{\omega_X}$ is
semistable.}

\begin{lemma}\label{l1}Let $X$ be a nonsingular
curve of genus $g\geq
2$ and $\sL$ be a (base point free) line bundle on $X$ such that
$\mbox{deg}~\sL > 2g$. Let $K_{\sL}$ be the syzygy bundle for the
evaluation map
$ H^0(X, \sL)\tensor \sO_X \by{} \sL $.
Then $K_{\sL}$ is stable.\end{lemma}
\begin{proof}Consider the short exact sequence
$$0\by{} K_{\sL} \by{} H^0(X, \sL)\tensor\sO_X \by{} \sL \by{} 0.$$
Now $h^0(X, \sL) = \deg~\sL+1-g$, since $H^1(X,\sL) = 0$, by
Serre duality. Therefore
$$\rank~K_{\sL} = \deg~\sL -g ~\mbox{and}~~\det
K_{\sL}^{\vee} =\sL,$$
 so that
 $\deg K_{\sL}^{\vee} = \deg~\sL$. This implies
$$\mbox{slope}~
(K_{\sL}^{\vee})
= \deg~\sL/(\deg~\sL-g) < 2.$$
Let $\sF$ be a quotient bundle of
$K_{\sL}^{\vee}$; then $\sF$ is generated by its global sections, and
$h^0(X,\sF) \geq r+1$ if rank~$\sF = r$ (otherwise it would
contradict the fact that $h^0(X, K_{\sL}) = 0$, because if $\sF$ is
trivial,
then so is $\sF^{\vee}$, and this would imply that
$h^0(X,K_{\sL})\geq r$). We choose (see [KR])  $W\subseteq
H^0(X,K_{\sL}^{\vee})$ such that
$\dim~W =
r+1$ and $W$ generates $\sF$; let
$$0\by{} \sM\by{} W\tensor\sO_X \by{} \sF \by{} 0 $$
be the corresponding short exact sequence. Then $\sM$ is a line bundle,
isomorphic to $\wedge^r\sF^{\vee}$, so that $\deg~\sF = \deg~\sM^{\vee}$.
 Note that $H^0(X, \sF^{\vee}) = 0$ as $H^0(X, K_{\sL})=0$ and therefore
$\dim~H^0(X, \sM^{\vee}) \geq r+1$.
\begin{enumerate}
\item Suppose $H^1(X, \sM^{\vee}) \neq 0$; then by Clifford's theorem
(chap~IV, Thoerem~5.4~[H])
$$\dim~H^0(X, \sM^{\vee}) -1 \leq (1/2)~(\deg~M^{\vee}).$$
 Hence $r\leq
(\deg~\sF)/2$.
This implies that $\mbox{slope}~\sF \geq 2 > 
\mbox{slope}~K_{\sL}^{\vee}$.
\item Suppose $H^1(X, \sM^{\vee}) = 0$. Then
$$H^0(X, \sM^{\vee}) =
\deg~\sM^{\vee} + 1 -g.$$ Hence $r+1 \leq \deg~\sF + 1-g$. This implies
$$\mbox{slope}~\sF\geq 1+g/r > 1 + g/(\deg~\sL- g) = \deg~\sL/(\deg~\sL- 
g) = 
\mbox{slope}~K_{\sL}^{\vee}.$$
\end{enumerate}
This proves the lemma.\end{proof}

\noindent{\it Proof of Proposition~\ref{p1}}:\quad Choose a nonsingular
curve $X$ of degree $d$ in $\bP^2_k$. 
We can deduce that  $$H^0(\bP^2_k, \sO_{\bP^2_k}(d)) \simeq H^0(X,
\sO_X(d))$$
from the following short exact sequence of sheaves of
$\sO_{\bP^2_k}$-modules
$$0\by{} \sO_{\bP^2_k}(-1)\by{} \sO_{\bP^2_k}(d) \by{} \sO_X(d)\by{} 0.$$

Let $\sL = \sO_{\bP^2_k}(d)\mid_X = \sO_X(d)$. 
Then we have a commutative diagram of exact sequences
$$\begin{array}{ccccccccc}
0 &\by{} & \sV_d\mid_X & \by{} & H^0(\bP^2_k, \sO_{\bP^2_k}(d))\tensor
\sO_X
&
\by{} &
\sO_{\bP_k^2}(d)\mid_X & \by{} &  0\\
  &      &  &    &        \downarrow{\cong} & &
\downarrow{\cong}  & &\\
0 & \by{} & K_{\sL} &  \by{} &
H^0(X, \sL)\tensor\sO_X & \by{} & \sL & \by{} & 0.
\end{array}$$
This implies that $\sV_d\mid_X\cong K_{\sL}$. 
But $\deg{\sL} = d^2 >  2(\mbox{genus}~X)$. Therefore, by 
Lemma~\ref{l1}, the bundle $K_{\sL}$ is stable on
$X$. Hence the bundle $\sV_d$ is stable on ${\bP_k^2}$. This
proves the
proposition. $\Box$

\section{The higher dimensional case}
Henceforth we assume that $n\geq 3$. Let $G = GL_{n+1}(k)$ and
let $P$ be the maximal parabolic group of $G$ given by
$$P= \left\{\left[\begin{array}{cc} 
g_{11} & *\\
0 & A \end{array}\right]
\in GL(n+1), ~~\mbox{where}~~A\in GL(n)\right\}.$$
Then there exists a canonical isomorphism 
$G/P\simeq \bP^n_k$. Recall that there is an equivalence of categories
between homogeneous $G$-bundles on $G/P$ and (finite 
dimensional)
$P$-modules.
The short exact sequence 

$$0 \by{} {\sV_d} \by{} H^0(\mathbf{P}_k^n, \sO_{{\mathbf
P}^n_k}(d))\tensor
\sO_{{\mathbf P}^n_k} \by{} \sO_{{\mathbf P}^n_k}(d) \by{} 0,$$
is naturally a sequence of homogeneous $G$-bundles, which 
corresponds to the short exact sequence of $P$-modules,

$$0\by{} V_d \by{} U_d\by{} W_d\by{} 0, $$
given as follows. Let $U_1$ and $V_1$ be $k$-vector spaces given by 
the basis $\{x_1, \ldots, x_n, z\}$ and $\{x_1, \ldots, x_n\}$
respectively.
Then $U_1$ is a $P$-module
such that if 
$$ g = \left[ \begin{array}{cc}
g_{11} & *\\
0 & A\end{array} \right], ~~~\mbox{where}~~A\in GL(n),$$
is an element of $P$ then 
 the representation $\rho: P \by{} {\rm{GL}}(U_1)$ is  defined as
follows (by matrix multiplication)

$$\rho(g)(z, x_1, \ldots, x_n) = 
\left[z, x_1,
\ldots, x_{n}\right]\cdot g^{-1}.$$

This gives canonical action of $P$ on $U_d = S^d(U_1)$ and on 
$$V_d = (S^1(V_1)\tensor z^{d-1}) \oplus (S^2(V_1)\tensor
z^{d-2}) \oplus  
\cdots \oplus S^d(V_1),$$ 
where elements of $U_d$ are homogeneous polynomials of degree $d$ in $x_1, 
\ldots, x_n, z$.

\begin{lemma}\label{r1} Let $\sW\subseteq \sV_d$ be a homogeneous
$G$-bundle and let $W\subset V_d$ be the corresponding $P$-module.
\begin{enumerate}
\item  Suppose
$f_0 + f_1z + \cdots + f_mz^m \in W$, where $f_i \in S^{d-i}(V_1)$. Then  
$f_iz^i \in W$, for all $i\geq 0$.
In other words, as $k$-vector spaces
 $$ W = W\cap (S^1(V_1)\tensor z^{d-1}) \oplus W\cap
(S^2(V_1)\tensor    
z^{d-2}) \oplus \cdots \oplus W\cap S^d(V_1).$$

 Moreover, 
\item if $i = i_0+i_1p+\cdots + i_mp^m $ denotes the $p$-adic expansion
of a positive integer $i$ and if $f\in S^{t_0}(V_1)$ is a homogeneous
polynomial such that 
$(f)z^i \in W$ (in particular $t_0+i = d$) then 
$$\begin{array}{lcl}
W & \supseteq & \phi\left[f\tensor \{S^{i_0}(V_1)\oplus
(S^{i_0-1}(V_1)\tensor 
z) \oplus
\cdots \oplus z^{i_0}\}\tensor \right.\\
& &  F^*\{S^{i_1}(V_1)\oplus (S^{i_1-1}(V_1)\tensor z) \oplus \cdots
\oplus z^{i_1}\}\tensor\cdots \\
& & \left.\tensor F^{m*}\{S^{i_m}(V_1)\oplus
(S^{i_m-1}(V_1)\tensor z)\oplus\cdots\oplus z^{i_m}\}\right],\end{array}$$
 where, for a positive integer $j$ such that  $j = j_0+j_p+\cdots+k_mp^m$, 
with the condition that  $0\leq j_k \leq i_k$, the map
 $$\phi: 
S^{t_0}(V_1)\tensor_{j=0}^{m}F^{j*}\left[\bigoplus_{k_j=0}^{i_j}
(S^{i_j-k_j}(V_1)\tensor z^{k_j})\right]
\rightarrow 
\bigoplus_{j=0}^{m}\bigoplus_{k_j=0}^{i_j}(S^{t_0+(i_j-k_j)p^j}(V_1)\tensor 
z^{i-(i_j-k_j)p^j}) $$ is 
the canonical map mapping to $V_d$, and 
$F^t$
denote
the $t^{th}$-iterated
Frobenius morphism and, for a vector-space $U$ generated  by 
$\{u_1, \ldots, u_t\}$, the vector-space $F^{i*}(U)$ is 
generated  by  $\{u_1^{p^i}, \ldots, u_t^{p^i}\}$.                
\end{enumerate}\end{lemma}
\begin{proof} The first part follows from the fact that the diagonal Torus 
group $T
\subseteq GL(n+1)$ is contained in $P$.

To prove the second part of the lemma, by looking at possible monomials 
occuring on the right side of $\supseteq $, we see that it is enough 
to prove the following: 
Let $ k < i $ be a  nonnegative integer so that if we have $k = 
k_0+k_1p+\cdots + k_mp^m$ with 
$0\leq k_j \leq i_j$, for $0\leq  j \leq m$.
Let $x_1^{T_1}\cdots x_n^{T_n}\in S^{i-k}(V_1)$ be a monomial 
with 
$$T_j  = t_{0j}+t_{1j}p+ \cdots +t_{mj}p^j, ~~~\mbox{where},~~~ 0\leq
t_{ij}\leq p-1  $$
such that  
$$(t_{01}+\cdots + t_{0n})+(t_{11}+\cdots+t_{1n})p+\cdots + 
(t_{m1}+\cdots +t_{mn})p^m = i-k,$$
where $ t_{j1} + \cdots + t_{jn} = i_j-k_j$. 
Then $(f)(x_1^{T_1}\cdots x_n^{T_n})z^{k} \in W$.

As $W$ is a $P$-module, $(f)z^i \in W$ implies that  
 $(f)(bz+a_1x_1+\cdots + a_nx_n)^i \in  W$,
for every $(b, a_1, \ldots, a_n)\in (k\setminus \{0\})\times k^n$.
Let $y = a_1x_1+\cdots + a_nx_n$. Now
  $(f)(bz+y)^i\in W$ implies that 
$$(f)\left[\binom{i}{1}(bz)^{i-1}y + \cdots + \binom{i}{i-1}(bz)y^{i-1} +
y^i\right] \in
W.$$
 
Hence, as argued in part~(1) of the lemma, we have
$(f)\binom{i}{k}y^{i-k}z^k\in W$, for
every $0 \leq k\leq i$.
Now $$ \binom{i}{k} = \frac{(i-k+k)\cdots (i-k+1)}{k(k-1)\cdots 1},$$
where the terms, divisible by $p$, in the numerator are 
$$\{ (i-k-(i_0-k_0)+lp)\mid 1\leq l \leq k_1+\cdots +k_mp^{m-1}\}$$ 
and in the denominator are
$$\{ lp \mid 1\leq l \leq k_1+\cdots +k_mp^{m-1}\}.$$

Hence if $k_j \leq i_j$, for all $0\leq j \leq m$,  then
${\rm{g.c.d.}}(\binom{i}{k}, p) = 1$ which
implies
$(f)y^{i-k}z^k \in W$.
That is, $(f)(a_1x_1+\cdots + a_nx_n)^{i-k}z^k \in W$,  for every
$(a_1,\ldots, a_n)\in k^n$. 

Now let $y_1 = a_2x_2+\cdots +a_nx_n$, so that  
$(f)(y_1+a_1x_1)^{i-k}z^k \in W$. Since  $T_1\leq i-k$ 
and $t_{j1} \leq i_j-k_j$, where 
$$t_{01}+t_{11}p+\cdots+t_{m1}p^m = T_1
~~\mbox{and}~~ (i_0-k_0)+(i_1-k_1)p+\cdots (i_m-k_m)p^m = i-k,$$ such
that $0\leq t_{ij}, i_j-k_j \leq p-1$.
 Hence, by a similar argument with a binomial expansion,  we have
$(f)y_1^{i-k-T_1}x_1^{T_1}z^k \in
W$. Iterating the arguement we deduce that $(f)x_1^{T_1}\cdots
x_n^{T_n}\in
W$. This completes the proof of the claim (and hence of the lemma).
\end{proof}
\vspace{5pt}

Now throughout this paper, we fix a positive integer $d$ with its $p$-adic
expansion  
$ d = a_0 + a_1p + \cdots + a_mp^m$ and 
we also fix a nonzero homogeneous $G$-subbundle
$\sW\subseteq \sV_d$ given by the corresponding $P$-module 
$W\subset V_d$. We would denote 
$$ W(i) = W\cap (S^{d-i}(V_1)\tensor z^i).$$  
By Lemma~\ref{r1},
$ W = \bigoplus_{i=0}^{d-1}W(i),$ 
such that $\sW$ has a filtration by $G$-subbundles
$\sF_0\subseteq \sF_1\subseteq \cdots \sF_{d-1} = \sW,$
where 
$\sF_i$ is the vector bundle associated to the $P$-submodule 
$W(0)\oplus W(1) \oplus \cdots \oplus W(i) \subseteq W$.
In particular, the vector-subspace $W(i)$ has the {\em canonical 
$P$-subquotient} module structure with associated homogeneous $G$-bundle 
$\sW(i)$, so
that there is a $G$-equivariant isomorphism
$$\mbox{gr}(\sW) \cong \bigoplus_{i=0}^{d-1} \sW(i).$$

For a $k$-vector space $V$, the number $|V|$ denotes the dimension 
of $V$.

\begin{remark}\label{r2} Let $i, j < d$ 
be two positive integers such 
that  $i = i_0+i_1p+\cdots
+i_mp^m$ and $j = j_0+j_1p+\cdots + j_mp^m$ with the condition that   
$0\leq j_k 
\leq i_k \leq p-1$, for every $k \geq 0$. Then,  
by part~(2) of 
Lemma~\ref{r1}, 
$$W(i)\neq 0 \implies W(j)\neq 0.$$

Moreover, if $W(i) = B_i\tensor z^i$, where $B_i \subseteq S^{d-i}(V_1)$,
then 
$$W(j)\supseteq \phi\left[B_i\tensor S^{i_0-j_0}(V_1)\tensor
F^{*}(S^{i_1-j_1}(V_1))\tensor \cdots \tensor
F^{*m}(S^{i_m-j_m}(V_1))\right]\tensor z^j,$$
where $$\phi:S^{d-i}(V_1)\tensor S^{i_0-j_0}(V_1)\tensor
F^{*}(S^{i_1-j_1}(V_1)\tensor \cdots \tensor
F^{*m}(S^{i_m-j_m}(V_1))\rightarrow S^{d-j}(V_1)$$ is the canonical map.   
\end{remark}

\begin{lemma}\label{l7}If $W(i) \neq 0$ and if $d-i = (t_0+ t_1p+\cdots +
t_mp^m)$, where 
$0\leq t_i \leq p-1$ then  
$$ S^{t_0}(V_1)\tensor F^{*}(S^{t_1}(V_1))\tensor \cdots \tensor
F^{*m}(S^{t_m}(V_1))\tensor z^i\subseteq W.$$\end{lemma}
\begin{proof}$W(i)\neq 0$ means $(S^{d-i}(V_1)\tensor z^i) \cap W \neq 
0$. Therefore $W = W_1\tensor z^i$, for some nonzero $SL(n)$-submodule 
$W_1$
of $S^{d-i}(V_1)$ (with 
respect to the induced action on $S^{d-i}(V_1)$ coming from the canonical 
action of $SL(n)$ on $V_1$. But (see [B]), for the 
$p$-adic 
expansion $t_0+t_1p+\cdots + t_mp^m = d-i$ of the integer $d-i$, the 
$SL(n)$-module  
$ S^{t_0}(V_1)\tensor F^{*}(S^{t_1}(V_1))\tensor \cdots \tensor
F^{*m}(S^{t_m}(V_1))$ is the smallest $SL(n)$-submodule of $S^{d-i}(V_1)$. 
In particular, it is contained in $W_1$. 
This proves the lemma.\end{proof}

\begin{remark}\label{r5}
For a vector-bundle $\sV$ on ${\bP^n_k}$, with determinant $\det(\sV) =
\sO_{\bP^n_k}(m)$, we define $\deg(\sV) = m$ and 
  $\mu(\sV) = \deg~(\sV)/\rank(\sV)$. 
We note that  
$$\deg~\sW = \sum \deg~\sW(i).$$ 
Since
$\sV_1$ is a semistable vector bundle on $\bP^n_k$, 
by Theorem~2.1 of [MR],
the vector bundle $S^{d-i}(\sV_1)$ is  semistable  on $\bP^n_k$. Hence 
$$-\deg~\sW(i) \geq -\mu(S^{d-i}(\sV_1)\tensor \sO(i))|W(i)|.$$ 
\end{remark}

Now we prove a series of lemmas before coming to the main result.

\begin{lemma}\label{l8} If $d = a_0 < p$ then, for $0\subsetneqq \sW 
\subsetneqq 
\sV_d$, we have  $-\deg~\sW \geq a_0 $.
In particular  $\mu(\sW) < \mu(\sV_d)$.
\end{lemma}
\begin{proof} For $d = a_0\leq p-1$, 
 the $P$-module $V_{a_0}$ is filtered by the 
subquotients isomorphic to  $S^{d-i}(V_1)\tensor z^{i}$,
where $ 0 \leq i < a_0 $. If $i_0$ is the 
largest integer
with the property that $ W(i_0) \neq 0$ then,
by Remark~\ref{r2} and Lemma~\ref{l7}, it follows that  
 $W = \sum_{j=0}^{i_0} S^{d-j}(V_1)\tensor z^{j}$, as 
$P$-module.
Therefore 
$$\begin{array}{lcl}
-\deg~\sW & = & \sum_{j=0}^{i_0} -\deg~(S^{d-j}(\sV_1)\tensor
z^{j})\\
&  =  & -\sum_{j=0}^{i_0} \mu~(S^{d-j}(\sV_1)\tensor
z^{j})|S^{d-j}(V_1)|\\
& =  & \frac{1}{n} \sum_{j=0}^{i_0} 
\left[(d-j)-nj\right]|S^{d-j}(V_1)|\\
&  = & (i_0+1)|S^{a_0-i_0-1}(V_1)| \geq a_0,\end{array}$$
where the second last equality follows from the fact that, 
for an integer $a\geq 0$, we have 
 $$(a+1)|S^{a+1}(V_1)| = n(|S^{a}(V_1)|+\cdots + |S^1(V_1)|+|S^0(V_1)|).$$
This proves the lemma.\end{proof}
 
In the rest of this section, we assume that the integer $d$ has the 
$p$-adic 
expansion $d = a_0+a_1p+\cdots +a_mp^m$ such that $a_0$ and $a_m$ are 
nonzero integers.
\begin{remark}\label{r3}
For $i_0+\cdots+ i_mp^m < a_0+\cdots +
a_mp^m$ (where $0\leq i_0, \ldots, i_m \leq p-1$),  let
$$W(i_0+i_1p+\cdots +i_mp^m) = W_{i_0,\cdots, i_m} = W\cap
\left[S^{d-(i_0+\cdots+i_mp^m)}(V_1)\tensor
z^{i_0+\cdots+i_mp^m}\right] $$
be the subspace with canonical $P$-(subquotient) structure and 
let  $\sW_{i_0,\cdots, i_m}$ be the associated $G$-bundle. Then, by
Lemma~\ref{r1}, 

$$\mbox{gr}~\sW =\bigoplus_{(i_0,\ldots, i_m)\in C_0(\sW)\cup\cdots \cup
C_m(\sW)}\sW_{i_0,\cdots,
i_m},$$
where 
$$\begin{array}{lcl}
C_0(\sW)  &  = & \{(i_0, a_1, \ldots , a_m)\mid 0\leq i_0 < a_0
~\mbox{and}~~ 
W_{i_0,\cdots, a_m} \neq 0\}\\
 & \vdots  & \\
C_j(\sW)  &  = & \{(i_0, \ldots, i_j, a_{j+1}, \ldots , a_m)\mid 0\leq i_j
< a_j~~\mbox{and}~~W_{i_0,\cdots, a_m} \neq 0 \}\\
 & \vdots  & \\
C_{m-1}(\sW)  &  = & \{(i_0, \ldots, i_{m-1}, a_m)\mid 0\leq i_{m-1}
< a_{m-1},~~\mbox{and}~~W_{i_0,\cdots, a_m} \neq 0 \}\\
C_m(\sW)  &  = & \{(i_0, \ldots, i_{m-1}, i_m)\mid 0\leq i_m <a_m, 
~~\mbox{and}~~W_{i_0,\cdots, i_m} \neq 0 \}\\
\end{array}$$
 Note that if $a_j = 0$, for some $j$ then $C_j(\sW) = \phi $.
Now 
$$\begin{array}{lcl}
-\mu(\sV_d)|W| & = & \displaystyle{\frac{d}{|V_{d}|}}
\left[\sum_{\{(i_0,\ldots, i_m)\in
C_m(\sW)\}}|\sW_{i_0,\ldots,
i_m}| + \sum_{\{(i_0,\ldots, i_m)\in C_0(\sW)\cup \cdots \cup 
C_{m-1}(\sW)\}}|\sW_{i_0,\ldots, i_m}|\right]\\
 & <  & (n |C_m(\sW)|) + (|C_0(\sW)|+ \cdots +
|C_{m-1}(\sW)|),\end{array}$$

where the last inequality follows, as
\begin{enumerate}
\item  for any  positive integer $a\geq 1$,  we
have 
$$\frac{a}{|V_a|}|S^a(V_1)| = \frac{a|S^a(V_1)|}{h^0(\sO(a))-1} = 
\frac{na|S^a(V_1)|}{(a+1)|S^{a+1}(V_1)|-n} < n.$$ 
Therefore $\{i_0,\ldots, i_m\} \in C_m(\sW) \implies 
\frac{d}{\sV_d}|\sW_{i_0,\ldots,i_m}| < n$.
\item the canonical inclusion (but not a  surjection)
$$F^{*m}S^{a_m}(V_1)\tensor
S^{a_0+\cdots + a_{m-1}p^{m-1}}(V_1) \hookrightarrow
 S^{a_0+\cdots + a_mp^m}(V_1)$$
 implies
 that 
$n|S^{a_0+\cdots+a_{m-1}p^{m-1}}(V_1)| <
|S^{a_0+\cdots+a_{m}p^{m}}(V_1)|$. Now if 
$$\{i_0, \ldots, i_m\} \in
C_0(\sW)\cup \cdots \cup C_{m-1}(\sW), ~~~\mbox{then}~~i_m = a_m,$$
which implies that $|W_{i_0,\ldots,
i_m}| \leq |S^{a_0+\cdots + a_{m-1}p^{m-1}}(V_1)|$. This implies that 
$\frac{d}{|V_d|}|\sW_{i_0,\ldots,i_m}| < 1.$ \end{enumerate}
\end{remark}

\begin{lemma}\label{l6}If $ W(a_0) = 0$ then $\mu(\sW) < \mu(\sV_d)$.
\end{lemma}
\begin{proof}Consider the following diagram of $G$-bundles,
$$\begin{array}{c}
 0\\
\downarrow\\
 0\rightarrow\sV_{a_0}\tensor F^{*}\sV_{a_1+\cdots+a_{m}p^{m-1}} 
\rightarrow H(a_0)\tensor
F^*\sV_{a_1+\cdots+a_mp^{m-1}}\rightarrow\sO(a_0)\tensor
F^*\sV_{a_1+\cdots+a_mp^{m-1}}\rightarrow 0 \\
 \downarrow\\
 \sV_{a_0+\cdots + a_mp^m}\\
 \downarrow\\
0 \longrightarrow \sV_{a_0}\tensor \sO(a_1+\cdots+a_mp^m)\longrightarrow
\coker(f) \longrightarrow
\oplus~\sO_X \longrightarrow 0,\hspace{3cm}\\
\downarrow \\
0
\end{array}$$

where $H(a_0) = H^0(\sO(a_0))$ and $f:H(a_0)\tensor
F^*(\sV_{a_1+\cdots+a_mp^{m-1}})
 \by{} \sV_{a_0+\cdots+a_mp^m}$ is the canonical map.
We note that, if $\sV'$ denotes the homogeneous bundle 
$\sV' =$ kernel of the canonical composite map 
$$\sV_{a_0+\cdots
a_mp^m}\by{}\coker(f) \by{} \oplus~\sO_X $$
and $V' \subset V_{a_0+\cdots +a_mp^m}$ the corresponding $P$-submodule, 
 then 
$$-\deg{\sW} \geq -\deg~\sW', ~\mbox{where}~W' = W\cap V'.$$
Since $ W(a_0) = 0$,
we have
 $$\sW'\subseteq 
\sV_{a_0}\tensor H^0(\sO(a_1+\cdots+a_mp^{m-1}))^{(p)},$$
which is a semistable vector bundle over $\bP^n_k$, as $\sV_{a_0}$
is semistable (see Lemma~\ref{l8}). Therefore
$-\deg~\sW'\geq \frac{a_0}{\sV_{a_0}}|~\sW'|$.

We remark that $|C_j(\sW')| = |C_j(\sW)|$, 
where we define $C_j(\sW)$ and
$C_j(\sW')$ as in Remark~\ref{r3}: 
Note $W_{i_0,\ldots,i_m}\neq 0$ implies that $i_0 < a_0$ as $W(a_0) = 0$. Let 
$$ j_0 + \cdots + j_mp^m = d-(i_0+\cdots +i_mp^m),~~~\mbox{where}~~
  0\leq j_t \leq p-1,$$
then $j_0 < a_0$. By Lemma~\ref{l7}, 
$$S^{j_0}(V_1)\tensor \cdots \tensor F^{*m}S^{j_m}(V_1)\tensor 
z^{i_0+\cdots +i_mp^m}\subseteq W_{i_0, \ldots, i_m}.$$
Therefore
$$S^{j_0}(V_1)\tensor \cdots \tensor F^{*m}S^{j_m}(V_1)\tensor 
z^{i_0+\cdots +i_mp^m} \subseteq
W_{i_0, \ldots, i_m} \cap V'_{i_0,\ldots, i_m} = W'_{i_0, \ldots,
i_m}.$$
In particular $W'_{i_0,\ldots,i_m}\neq 0$ and $|C_j(\sW')| = |C_j(\sW)|$. 
\linebreak Moreover $|W'_{i_0,\ldots, i_m}|\geq 
|S^{j_0}(V_1)|\cdots |F^{*m}S^{j_m}(V_1)|$.

But
$$\begin{array}{lcl}
 -\deg~\sW'& \geq  & 
\displaystyle{\frac{a_0}{|V_{a_0}|}}\left[\sum_{\{(i_0, \ldots, i_m)\in
C_m(\sW')\}}|W'_{i_0,\cdots,i_m }|
 +
\sum_{\{(i_0, \ldots, i_{m})\in C_0(\sW')\cup\cdots\cup  C_{m-1}(\sW')\}} 
|W'_{i_0,\cdots, a_m}|\right] \end{array}$$ 
If $(i_0,\ldots, i_m)\in C_m(\sW')$ then $i_0<a_0$ and $i_m < a_m$. 
Therefore, for the $p$-adic expansion 
$$j_0+\cdots +j_mp^m = a_0+\cdots + a_mp^m -(i_0+\cdots +i_mp^m),$$
 we 
have 
$j_0 > 0$ and $j_1p+\cdots + j_mp^m > 0$. In particular
$$|W'_{i_0, \ldots, i_m}| \geq |S^{j_0}(V_1)|\cdots 
|S^{j_m}(V_1)|\geq n|S^{j_0}(V_1)| = n|S^{a_0-i_0}(V_1)| .$$
Now 
$$\begin{array}{lcl}
(\ast) & := & \displaystyle{\frac{a_0}{|V_{a_0}|}}
\sum_{\{(i_0, \ldots, i_m)\in C_m(\sW')\}}|W'_{i_0,\cdots,i_m }|\\
& = & \displaystyle{\frac{a_0}{|V_{a_0}|}}\sum_{\{(i_1, \ldots, i_m)\mid 
i_m<a_m\}}{}
\sum_{\{k\mid (k,i_1,\ldots, i_m)\in
C_m(\sW')\}}|W'_{k, i_1,\cdots,i_m }|\\
& \geq & 
 \displaystyle{\frac{a_0}{|V_{a_0}|}}\sum_{\{(i_1, \ldots, i_m)\mid 
i_m<a_m\}}
\sum_{\{k\mid (k,i_1,\ldots, i_m)\in
C_m(\sW')\}}n|S^{a_0-k}(V_1)|.\end{array}$$

Let 
$I_0(i_1, \ldots, i_m) =  0$ if  $(k, i_1, \ldots, i_m) \not\in 
C_m(\sW')$  for all  $k$, otherwise define 

 $$I_0(i_1, \ldots, i_m) = \mbox{max}~\{k_0+1\mid (k_0, i_1,
\ldots, i_m) \in C_m(\sW')\}.$$ 

Then from the inequality 
$$\frac{a_0}{|V_{a_0}|}\left(|S^{a_0}(V_1)|+ \cdots +
|S^{a_0-k}(V_1)|\right) \geq k+1,$$  
for any $0\leq k \leq a_0-1$, it follows that 
$$(\ast) \geq 
n\sum_{\{(i_1,\ldots, i_m)\mid i_m <a_m\}} I_0(i_1, \ldots, i_m) 
= n|C_m(\sW')|.$$
Similarly, we can argue that 
$$\begin{array}{lcl}
\displaystyle{\frac{a_0}{|V_{a_0}|}}
\sum_{\{(i_0, \ldots, a_{m})\in C_0(\sW')\cup \cdots \cup C_{m-1}(\sW')\}} 
|W'_{i_0,\cdots, a_m}| & \geq & 
\displaystyle{\sum_{\{(i_1,\ldots, i_{m-1}, a_m)\}}}I_0(i_1, \ldots, 
i_{m-1},
a_m)\\
&   = & |C_0(\sW')|+ \cdots
+ |C_{m-1}(\sW')|.\end{array}$$
 
This implies 
$$\begin{array}{lcl}
-\deg~\sW' & \geq & n|C_m(\sW')| + (|C_0(\sW')|+\cdots + 
|C_{m-1}(\sW')|)\\
 &  = & n|C_m(\sW)|+(|C_0(\sW)| + \cdots + |C_{m-1}(\sW)|).\end{array}$$

On the other hand, by Remerk~\ref{r3},
$$-\mu(\sV_d)|W|  <  (n|C_m(\sW)|) + (|C_0(\sW)|+\cdots + 
|C_{m-1}(\sW)|),$$
which implies that $-\mu(\sV_d)|W| \leq  -\deg~\sW $.
This proves the lemma. \end{proof}

\begin{lemma}\label{l2}If  $W(a_0+
\cdots + a_{m-1}p^{m-1})\neq 0$ then 
$$-\deg(\sW) \geq a_0 + \cdots +a_mp^m = -\deg~(\sV_d).$$
In particular $\mu(\sW) < \mu(\sV_d)$, if $\sW\subsetneqq \sV$.
\end{lemma}
\begin{proof} For any $a\in {\mathbb N}$, let $H(a) = H^0(\bP^n_k,
\sO_{\bP^n_k}(a))$ and let  $H(a)^{(p^t)} = 
F^{*t}(H(a)\tensor\sO_{\bP^n_k})$, where
$F^t$ is the $t^{th}$ iterated Frobenius morphism. 
Note that there is an exact sequence of $G$-bundles 
$$0\by{} F^{*t}(\sV_a)\by{} H(a)^{(p^t)}\by{} \sO_{\bP^n_k}(ap^t)\by{}
0.$$
Let $\delta$ denote the following tensor product map of $G$-bundles:
$$\tensor_{i = 0}^{m-1}H(a_i)^{(p^i)}\by{} \tensor_{i
=0}^{m-1}\sO_{\bP^n_k}(a_ip^i) = \sO_{\bP^n_k}(a_0+a_1p+\cdots +
a_{m-1}p^{m-1}).$$
We then have an induced commutative diagram of homogeneous
$G$-bundles, with exact rows and coloumns (the term $\oplus~\sO_{\bP^n_k}$
denotes
a certain trivial vector bundle, with a $G$-action),
 $$\begin{array}{ccc}
{} & {}  {} & {}  0\\
{} & {}  {} & {} \downarrow\\
 0 & {}  0 & {}  \ker~\delta \tensor \sO(a_mp^m)\\
\downarrow{} {} & \downarrow{} & {}  \downarrow{}\\
0\rightarrow{{\displaystyle{\tensor_{i=0}^{m-1}H(a_i)^{(p^i)}\tensor
F^{*m}\sV_{a_m}}}}
 \rightarrow &
\displaystyle{\tensor_{i=0}^{m-1}H(a_i)^{(p^i)}\tensor
H(a_m)^{(p^m)}}  \rightarrow & 
\displaystyle{\tensor_{i=0}^{m-1}H(a_i)^{(p^i)}\tensor
\sO(a_mp^m)}\rightarrow 0 \\
 \downarrow{f} &  \downarrow{} & {} \downarrow{} \\
 0\rightarrow\sV_{a_0+\cdots +a_mp^m}\rightarrow & 
H(a_0+\cdots + a_mp^m)\tensor
\sO_{\bP^n_k}\rightarrow & \sO(a_0+\cdots
+a_mp^m) \rightarrow 0\\
 \downarrow{} & {} \downarrow{}  & \downarrow{} \\
\coker~(f) &  \oplus~\sO_{\bP^n_k} & {}  0 \\
 \downarrow{}  & \downarrow{} &  {} \\
 0 & {} 0 & {}.
\end{array}$$
  This gives the following diagram of homogeneous
$G$-bundles
$$\begin{array}{c}
{}  0 {}\\
{} \downarrow {}\\
 0   \rightarrow \ker\delta\tensor F^{*m}\sV_{a_m}  \rightarrow 
\displaystyle{\tensor_{i=0}^{m-1}H(a_i)^{(p^i)}}\tensor
F^{*m}\sV_{a_m} \rightarrow   \sO(a_0+\cdot\cdot+a_{m-1}p^{m-1})\tensor
F^{*m}\sV_{a_m} \rightarrow 0 \\
{} \downarrow  {f}\\
{}   \sV_{a_0+\cdots + a_mp^m}  {} \\
{}  \downarrow  {}\\
0 \longrightarrow \ker\delta\tensor
\sO(a_mp^m) 
\longrightarrow
\coker(f) \longrightarrow 
\oplus~\sO_{{\mathbf
P}^n_k} \longrightarrow 0\hspace{2cm}\\
{}  \downarrow {}\\
{}  0  {}.
\end{array}$$
We note that, if $\sV'$ denotes the homogeneous $G$-bundle 
 given by $\sV' =$ kernel of the canonical composite map 
$\sV_{a_0+\cdots +
a_mp^m}\by{}\coker(f) \by{} \oplus~\sO_{{\mathbf
P}^n_k}, $
{\it, i.e.}, 
$$0\by{} \sV' \by{} \sV_{a_0+\cdots + a_mp^m} \by{} 
\oplus~\sO_{{\mathbf P}^n_k}\by{} 0,$$ 
Then
$$-\deg{\sW} \geq -\deg{\sW'},~~\mbox{where}~~W' = W\cap V',$$
$V'$ is the $P$-module associated to $\sV'$ and $\sW'$ is the
$G$-bundle associated to the $P$-module $W'$.
Let 
$${ A_{-1}} = k.(z^{a_0+\cdots +a_{m-1}p^{m-1}})$$
and, 
for $0\leq i_0\leq m-1$, let 
$A_{i_0} = F^{*{i_0}}V_{a_{i_0}}\tensor
z^{a_0+\cdots +\hat{a_{i_0}p^{i_0}}+\cdots + a_{m-1}p^{m-1}}$,
where 
$$z^{a_0+\cdots +\hat{a_{i_0}p^{i_0}}+\cdots + a_{m-1}p^{m-1}} = 
z^{a_0+\cdots + a_{m-1}p^{m-1} - a_{i_0}p^{i_0}}.$$
Inductively, we define 
$$ A_{i_0\ldots i_j}  =    
F^{*i_0}V_{a_{i_0}}\tensor \cdots \tensor F^{*i_j}V_{a_{i_j}}\tensor
z^{a_0+\cdots + \hat{a_{i_0}p^{i_0}}+ \cdots +\hat{a_{i_j}p^{i_j}}+ \cdots
+
a_{m-1}p^{m-1}},$$
where $0\leq i_0 <\cdots <i_j \leq m-1$ and
 $$z^{a_0+\cdots + \hat{a_{i_0}p^{i_0}}+ \cdots +\hat{a_{i_j}p^{i_j}}+ 
\cdots
+ a_{m-1}p^{m-1}} = z^{a_0+\cdots + a_{m-1}p^{m-1}-(a_{i_0}p^{i_0}+\cdots 
+ a_{i_j}p^{i_j})}.$$
We can write 
$$ V' = \bigoplus_{j=-1}^{m-1}\bigoplus_{i_0, \ldots, 
i_j}A_{i_0,\ldots, 
i_j}\tensor F^{m*}V_{a_m}
+ \bigoplus_{j=0}^{m-1}\bigoplus_{i_0, \ldots, i_j}A_{i_0,\ldots, 
i_j}\tensor z^{a_mp^m},$$
where by 
$$\bigoplus_{i_0, \ldots, i_j}~~\mbox{we  mean}~~~\bigoplus_{ \{0\leq i_0 
< 
\cdots <  
i_j \leq m-1\mid A_{i_0, \ldots, i_j} \neq 0\}}.$$
 Note 
that 
 $A_{i_0\ldots i_j} \neq 0$ if and only if  $a_{i_k} \neq 0$, for every 
$i_k\in\{i_0, \ldots,  i_j\}$.  We also note 
that  
$A_{i_0,\ldots, 
i_j}\tensor F^{m*}V_{a_m}$ and $A_{i_0,\ldots, i_j}\tensor z^{a_mp^m}$ 
have canonical $P$-subquotient module structure.
By Lemma~\ref{r1}, 
$$W' = \bigoplus_{-1\leq j \leq m-1} {\tilde B_j} \bigoplus_{0\leq j \leq 
m-1}{\tilde C_j},$$
where 
 $${\tilde B_j} = \bigoplus_{i_0,\ldots, i_j} B_{i_0, \ldots, 
i_j},~~\mbox{where}~~ B_{i_0, \ldots, i_j} = (A_{i_0, \ldots, i_j}\tensor 
F^{m*}V_{a_m})\cap 
W \subseteq A_{i_0, \ldots, i_j}\tensor F^{m*}V_{a_m} $$
and 
$${\tilde C_j} = \bigoplus_{i_0, \ldots, i_j}C_{i_0, \ldots, 
i_j},~~\mbox{where}~~ 
C_{i_0, \ldots, i_j} = (A_{i_0, \ldots, i_j}\tensor z^{a_mp^m})\cap W 
\subseteq A_{i_0, \ldots, i_j}\tensor z^{a_mp^m}.$$
Then  $B_{i_0, \ldots, i_j}$ and $C_{i_0, \ldots, i_j}$ have canonical   
$P$-subquotient module structures. Let $\sB_{i_0, \ldots, i_j}$ and 
$\sC_{i_0, 
\ldots, i_j}$ be the associated $G$-subbundles in $\sA_{i_0, \ldots, 
i_j}\tensor F^{m*}\sV_{a_m}$ and $\sA_{i_0,\ldots, i_j}\tensor 
\sO(a_mp^m)$, respectively. Moreover
 it follows that  $A_{i_0,\ldots, i_j} \neq 
0$ implies 
 $B_{i_0, \ldots, i_j} \neq 0$. The bundle $\sW'$ has a filtration by 
$G$-subbundles such that subquotients are isomorphic to $\sB_{i_0, \ldots, 
i_j}$ or to $\sC_{i_0,\ldots, i_j}$. 

Therefore 
$$-\deg~\sW \geq -\deg~{\tilde \sB_{-1}}-\deg~{\tilde \sB_{0}} + \cdots 
-\deg~{\tilde \sB_{m-1}}-
\deg~{\tilde \sC_{0}} + \cdots -\deg~{\tilde \sC_{m-1}},$$
where ${\tilde \sB_j} $ and ${\tilde \sC_j}$ are  the $G$-bundles 
associated to 
$P$-modules ${\tilde B_j}$ and ${\tilde C_j}$.

Henceforth, for a vector bundle $\sB$, we denote rank~$\sB$ as $|\sB|$.
Note that ${\sA_{i_0, \ldots, i_j}}\tensor F^{*m}\sV_{a_m}$ and 
${\sA_{i_0, \ldots, i_j}}\tensor \sO(a_mp^m)$
are semistable bundles on $\bP^n_k$; as by Lemma~\ref{l8}, for $0\leq a 
\leq p-1$, the bundle 
$\sV_{a}$  is semistable, and hence, by Theorem~2.1 of [MR], all 
Frobenius pullbacks and 
tensor
products of such
bundles are semistable.

Now, for $j\geq 0$, 
$$-\deg~{\tilde \sB_j} = -\sum_{i_0,\ldots, i_j}\deg~\sB_{i_0,\ldots, i_j} 
\geq 
-\sum_{i_0,\ldots, i_j}\mu(\sA_{i_0,\ldots, i_j})|\sB_{i_0,\ldots, i_j}|
 + \sum_{i_0,\ldots, i_j} 
\frac{a_mp^m}{|\sV_{a_m}|}|\sB_{i_0,\ldots, i_j}|$$
$$ = \sum_{i_0, \ldots, i_j} \left[-(a_0+\cdots 
+a_{m-1}p^{m-1})+a_{i_0}p^{i_0} + \cdots + a_{i_j}p^{i_j}\right]
|\sB_{i_0, \ldots, i_j}| $$
$$+\sum_{i_0, \ldots, i_j}\left[\frac{a_{i_0}p^{i_0}}{|\sV_{a_{i_0}}|} + 
\cdots +
\frac{a_{i_j}p^{i_j}}{|\sV_{a_{i_j}}|} 
\right]|{\sB_{i_0, \ldots, i_j}}|  
   +\sum_{i_0,\ldots,i_j}\frac{a_mp^m}{|\sV_{a_m}| }|{\sB_{i_0, 
\ldots, i_j}}|,$$
and, for $j\geq 1$,
$$-\deg~{\tilde \sC_j} = -\sum_{i_0,\ldots, i_j}\deg~\sC_{i_0,\ldots, i_j} 
\geq 
-\sum_{i_0,\ldots, i_j}\mu(\sA_{i_0,\ldots, i_j})|\sC_{i_0,\ldots, i_j}|
 - \sum_{i_0,\ldots, i_j}
a_mp^m|\sC_{i_0,\ldots, i_j}|$$
$$ = \sum_{i_0, \ldots, i_j} \left[-(a_0+\cdots 
+a_{m-1}p^{m-1})+a_{i_0}p^{i_0} + \cdots + a_{i_j}p^{i_j}\right]
|\sC_{i_0, \ldots, i_j}| $$
$$+\sum_{i_0, \ldots, i_j}\left[\frac{a_{i_0}p^{i_0}}{|V_{a_{i_0}}|} + 
\cdots +
\frac{a_{i_j}p^{i_j}}{|V_{a_{i_j}}|} 
\right]|{\sC_{i_0, \ldots, i_j}}|  
   - \sum_{i_0,\ldots,i_j} a_mp^m |{\sC_{i_0, 
\ldots, i_j}}|.$$
By construction.
${\tilde \sB_{-1}} = F^{m*}{\hat \sB_{-1}}\tensor \sO(a_0+\cdots 
+a_{m-1}p^{m-1})$, where ${\hat \sB_{-1}} \subseteq \sV_{a_m}$ is a 
$G$-subbundle. By Lemma~\ref{l8}, we have $-\deg~{\hat \sB_{-1}} \geq 
a_m$.
Therefore
$$-\deg~\sB_{-1} \geq a_mp^m - (a_0+\cdots + a_{m-1}p^{m-1})|{\tilde 
\sB_{-1}}|.$$
Similarly 
$${\tilde \sC_0} = \bigoplus_{i_0=0}^{m-1}\sC_{i_0} = 
\bigoplus_{i_0=0}^{m-1}F^{i_0*}{\hat\sC_{i_0}}\tensor \sO(a_0+\cdots 
{\hat a_{i_0}p^{i_0}}+\cdots+ a_mp^m),$$
where ${\hat \sC_{i_0}}$ is a $G$-subbundle of $\sV_{a_m}$.
Therefore 
$$ -\deg~{\tilde \sC_0} \geq \sum_{i_0=0}^{m-1}\left[-(a_0+\cdots
+a_{m-1}p^{m-1})+a_{i_0}p^{i_0}\right]|{ \sC_{i_0}}| 
+ \sum_{i_0=0}^{m-1}(a_{i_0}p^{i_0}\delta_{i_0} - a_mp^m|{\sC_{i_0}}|),$$
where $\delta_{i_0} = 1$, if $C_{i_0} \neq 0$, otherwise $\delta_{i_0} = 
0$.
        
\begin{claim}\label{cm1} \begin{enumerate}
\item $\displaystyle{\frac{a_mp^m}{|\sV_{a_m}|}}|\sB_{i_0,\ldots, i_j}| - 
a_mp^m|\sC_{i_0,\ldots, i_j}|
\geq 0$.
\item $\displaystyle{\frac{a_mp^m}{|\sV_{a_m}|}}|\sB_{i_0}|
+ a_{i_0}p^{i_0}\delta_{i_0} - 
a_mp^m|\sC_{i_0}| \geq a_{i_0}p^{i_0}$.
\item 
$\displaystyle{\frac{a_{i_k}p^{i_k}}{|\sV_{a_{i_k}}|}}|\sB_{i_0,\ldots,i_k,\ldots, 
i_j}| \geq a_{i_k}p^{i_k}|\sB_{i_0,\ldots,{\hat i_k},\ldots, i_j}|$
\item 
$\displaystyle{\frac{a_{i_k}p^{i_k}}{|\sV_{a_{i_k}}|}}|\sC_{i_0,\ldots,i_k,\ldots, 
i_j}| \geq a_{i_k}p^{i_k}|\sC_{i_0,\ldots,{\hat i_k},\ldots, i_j}|$,
\end{enumerate}
where $\sB_{i_0,\ldots,{\hat i_k},\ldots, i_j} = 
\sB_{i_0, \ldots, i_{k-1}, i_{k+1}, \ldots, i_j}$ and
$\sC_{i_0,\ldots,{\hat i_k},\ldots, i_j} = 
\sC_{i_0, \ldots, i_{k-1}, i_{k+1}, \ldots, i_j}$ 
\end{claim}
\noindent{\it Proof}: We note that $C_{i_0,\ldots, i_j} = {\hat 
C}_{i_0,\ldots, i_j} 
\tensor z^{a_0+\cdots +{\hat a_{i_0}p^{i_0}}+\cdots + {\hat 
a_{i_j}p^{i_j}}+\cdots+a_mp^m} \subseteq V'\cap W$,
where ${\hat C _{i_0,\ldots, i_j}} \subseteq F^{i_0*}V_{a_{i_0}}\tensor
\cdots \tensor F^{i_j*}V_{a_{i_j}}$ is a $P$-submodule.
By Remark~\ref{r2}, this implies that 
$${\hat C_{i_0,\ldots, i_j}}\tensor F^{m*}V_{a_m}\tensor 
z^{a_0+\cdots +{\hat a_{i_0}p^{i_0}}+\cdots + {\hat 
a_{i_j}p^{i_j}}+\cdots+a_{m-1}p^{m-1}} \subseteq V'\cap W,$$
therefore
 $${\hat C_{i_0,\ldots, i_j}}\tensor F^{m*}V_{a_m}\tensor 
z^{a_0+\cdots +{\hat a_{i_0}}+\cdots + {\hat 
a_{i_j}}+\cdots+a_{m-1}p^{m-1}}\subseteq B_{i_0,\ldots, i_j}.$$
Hence 
\begin{equation}\label{e6} |B_{i_0,\ldots, i_j}|\geq 
|{\hat C_{i_0,\ldots, i_j}}||V_{a_m}| = |C_{i_0,\ldots, i_j}||V_{a_m}|.
\end{equation}
This proves assertion~(1). 

Similarly one can check
that
$$|B_{i_0,\ldots, i_k, \ldots, i_j}| \geq 
|V_{a_{i_k}}||B_{i_0,\ldots,{\hat i_k},\ldots, i_j}|~~\mbox{and}~~
 |C_{i_0,\ldots, {\hat i_k}, \ldots, i_j}| \geq 
|V_{a_{i_k}}||C_{i_0,\ldots,i_k,\ldots, i_j}|.$$
Hence assertions (3) and (4) follow.

Now, assertion~(2) follows, by inequality~(\ref{e6}), if $|C_{i_0}| \neq 
0$.
Therefore we can assume that  $|C_{i_0}| = 0$. We can also assume that 
$a_{i_0} \neq 0$.
By Remark~\ref{r2}, 
$$W(a_0+\cdot+ a_{m-1}p^{m-1}) \neq 0 \implies
W(a_0+\cdots+ a_{i_0-1}p^{i_0-1}+a_{i_0+1}p^{i_0+1}+\cdots  
+a_{m-1}p^{m-1}) \neq 0.$$
 Therefore, by Lemma~\ref{l7}, 
$$F^{i_0*}S^{a_{i_0}}(V_1)\tensor F^{m*}S^{a_{m}}(V_1)\tensor 
z^{a_0+\cdots 
+{\hat a_{i_0}p^{i_0}}+\cdots +a_{m-1}p^{m-1}} \subseteq B_{i_0}.$$
Hence $|\sB_{i_0}|\geq |S^{a_{i_0}}(\sV_1)||S^{a_m}(\sV_1)|$. But, for 
any 
interger $a\geq 1$, we can check the inequality 
$$\frac{|S^{a}(\sV_1)|}{|\sV_{a}|}\geq \frac{n}{n+a}.$$
 Therefore
$$\begin{array}{lcl}
 \displaystyle{\frac{a_mp^m}{|\sV_{a_m}|}|\sB_{i_0}|} & \geq &  
\displaystyle{\frac{a_mp^m}{|\sV_{a_m}|}|S^{a_{i_0}}(\sV_1)||S^{a_m}(\sV_1)|}\\
& \geq &
\displaystyle{\frac{na_mp^m}{(n+a_m)}|S^{a_{i_0}}(\sV_1)| ~~\geq~~ 
\frac{na_mp}{(n+a_m)}p^{m-1}|S^{a_{i_0}}(\sV_1)|~~\geq~~ 
a_{i_0}p^{i_0}},\end{array}$$
where the last inequality follows as $i_0\leq m-1$.
This proves the assertion~(2), and hence the claim.

Therefore 

$$-\deg~{\tilde \sB_{-1}} -\deg~{\tilde\sB_0}-\deg~{\tilde\sC_0} = 
[a_mp^m] 
+ 
\left[-(a_0+\cdots +a_{m-1}p^{m-1})|{\tilde\sB_{-1}}|+
\sum_{i_0=0}^{m-1}\frac{a_{i_0}p^{i_0}}{|V_{a_{i_0}}|}|{\sB_{i_0}}|\right]$$

$$ +\sum_{i_0=0}^{m-1}\left[-(a_0+\cdots
+a_{m-1}p^{m-1})+a_{i_0}p^{i_0}\right]\left[|{\sB_{i_0}}|+|\sC_{i_0}|\right]
+\sum_{i_0=0}^{m-1}\left[\frac{a_mp^m}{|V_{a_m}|}|\sB_{i_0}|
-a_mp^m|\sC_{i_0}|+a_{i_0}p^{i_0}\delta_{i_0}\right].$$

Now, by assertions~(3) and ~(2) of Claim~\ref{cm1}, applied respectively 
to 
the terms in 
second and fourth bracket above, we get 

$$-\deg~{\tilde \sB_{-1}} 
-\deg~{\tilde\sB_0}-\deg~{\tilde\sC_0}\hspace{7cm} $$
$$ \hspace{2cm}~\geq 
a_0 
+\cdots + a_mp^m +\sum_{i_0=0}^{m-1}\left[-(a_0+\cdots
+a_{m-1}p^{m-1})+a_{i_0}p^{i_0}\right]\left[|\sB_{i_0}|+|\sC_{i_0}|\right].$$
Moreover, applying assertion~(1) of Claim~\ref{cm1} , we get
$$-\deg~{\tilde \sB_1}-\deg~{\tilde\sC_1} = 
 \sum_{i_0,i_1}\left[-(a_0+\cdots
+a_{m-1}p^{m-1})+a_{i_0}p^{i_0} + a_{i1}p^{i_1}\right] +$$ 
$$ \displaystyle{\sum_{i_0,i_1}\left[\frac{a_{i_0}p^{i_0}}{|V_{a_{i_0}}|} + 
\frac{a_{i_1}p^{i_1}}{|V_{a_{i_1}}|}\right]\left[|\sB_{i_0i_1}| +
|\sC_{i_0i_1}|\right]}.$$

\vspace{5pt}

\begin{claim}\label{cm2} $$\sum_{i_0,i_1}
\left[\frac{a_{i_0}p^{i_0}}{|V_{a_{i_0}}|} + 
\frac{a_{i_1}p^{i_1}}{|V_{a_{i_1}}|}\right]\left[|\sB_{i_0i_1}| +
|\sC_{i_0i_1}|\right] 
 \geq  \sum_{i_0=0}^{m-1}\left[(a_0+\cdots
+a_{m-1}p^{m-1})-a_{i_0}p^{i_0}\right]\left[|\sB_{i_0}|+|\sC_{i_0}|\right].$$
\end{claim}
\noindent{\it Proof}:\quad~By assertion~(3) and (4) of 
Claim~\ref{cm1}, we have 

$$\begin{array}{lcl}
\mbox{left hand side} & \geq & \displaystyle{\sum_{i_0,i_1}\left[
a_{i_0}p^{i_0}(|B_{i_1}|+|C_{i_1}|)\right]+
\sum_{i_0,i_1}\left[a_{i_1}p^{i_1}(|B_{i_0}|+|C_{i_0}|)\right]}\\
{} & = & \displaystyle{\sum_{i_0,i_1}\left[
a_{i_0}p^{i_0}|B_{i_1}|+a_{i_1}p^{i_1}|B_{i_0}|)\right]+
\sum_{i_0,i_1}\left[
a_{i_0}p^{i_0}|C_{i_1}|+a_{i_1}p^{i_1}|C_{i_0}|)\right]}\\
&  = & \displaystyle{\sum_{i_0=0}^{m-1}\left[(a_0+\cdots
+a_{m-1}p^{m-1})-a_{i_0}p^{i_0}\right]\left[|\sB_{i_0}|+|\sC_{i_0}|\right]}.
\end{array}$$
This proves the claim.

In particular
$$-\deg~{\tilde\sB_{-1}}-\deg~{\tilde\sB_0}-\deg~{\tilde\sC_0}-
\deg~{\tilde\sB_1}-\deg~{\tilde\sC_1} 
\geq  \hspace{7cm}$$
$$ (a_0 + \cdots +a_mp^m) + 
\displaystyle{\sum_{i_0,i_1}[-(a_0+\cdots
+a_{m-1}p^{m-1})+a_{i_0}p^{i_0} + a_{i1}p^{i_1}][|\sB_{i_0i_1}| +
|\sC_{i_0i_1}|]}.$$

By assertion~(1) of Claim~\ref{cm1}, for $j\geq 2$, we have
$$-\deg{\tilde\sB_j}-\deg~{\tilde\sC_j}
 =   \sum_{i_0, \ldots, i_j}\left[\frac{a_{i_0}p^{i_0}}{|V_{a_{i_0}}|} 
+ 
\cdots + \frac{a_{i_j}p^{i_j}}{|V_{a_{i_j}}|} 
\right]\left[|{\sB_{i_0, \ldots, i_j}}|+|\sC_{i_0, \ldots, i_j}|\right] 
 $$
$$ + \sum_{i_0, \ldots, i_j} \left[-(a_0+\cdots 
+a_{m-1}p^{m-1})+a_{i_0}p^{i_0} + \cdots + a_{i_j}p^{i_j}\right]
\left[|\sB_{i_0, \ldots, i_j}|+|\sC_{i_0, \ldots, i_j}| \right]$$ 

Now one can check (similar to the proof of Claim~\ref{cm2}) 
that 
$$\sum_{i_0, \ldots, i_j}\left[\frac{a_{i_0}p^{i_0}}{|V_{a_{i_0}}|} + 
\cdots + \frac{a_{i_j}p^{i_j}}{|V_{a_{i_j}}|} 
\right]\left[|{\sB_{i_0, \ldots, i_j}}|+|\sC_{i_0, \ldots, i_j}|\right]  
  \geq \hspace{5cm} $$
$$ \sum_{i_0, \ldots, i_{j-1}} \left[(a_0+\cdots 
+a_{m-1}p^{m-1})-(a_{i_0}p^{i_0} + \cdots + a_{i_{j-1}}p^{i_{j-1}})\right]
\left[|\sB_{i_0, \ldots, i_{j-1}}|+|\sC_{i_0, \ldots, i_{j-1}}| \right].$$

 Now it follows that, for $0\leq j \leq m-1$,
$$-\deg~{\tilde\sB_{-1}} - (\deg~{\tilde\sB_{0}} + \deg~{\tilde\sC_0} )- 
\cdots -
(\deg~{\tilde\sB_j}+\deg~{\tilde\sC_j})\geq (a_0+\cdots 
+a_mp^m)\hspace{5cm} $$
$$ + \sum_{i_0,\ldots,i_j}\left[-(a_0+\cdots
+a_{m-1}p^{m-1})+a_{i_0}p^{i_0} + \cdots + a_{i_j}p^{i_j}
\right]\left[|\sB_{i_0\cdots i_j}| +
|\sC_{i_0\cdots i_j}|\right].$$
Therefore, for $j = m-1$ 
$$-\deg~{\tilde\sB_{-1}} - (\deg~{\tilde\sB_0} + \deg~{\tilde\sC_0})- 
\cdots -
(\deg~{\tilde\sB_{m-1}}+\deg~{\tilde\sC_{m-1}})\geq 
a_0+\cdots +a_mp^m.$$
This proves the lemma.
\end{proof}

We have immediate corollory of this lemma.

\begin{cor}\label{c3}Suppose $d = a_0 + a_mp^m$ is the $p$-adic expansion 
of the positive integer $d$. Suppose 
$\sW\subset 
\sV_d$ is a $G$-subbundle such that $W(a_0)\neq 0$. Then  
$-\deg(\sW) \geq  d = a_0+a_mp^m $. \end{cor}

\begin{lemma}\label{l3}Suppose $\sW\subset 
\sV_d$ is a $G$-subbundle such that  $W(a_0)\neq 0$.
 Let $i$ be the integer such 
that  $ W(a_0+\cdots
+a_{i-1}p^{i-1}) \neq 0$
and $W(a_0+\cdots + a_ip^i) = 0$.  We 
further assume that $a_{i+1}, \ldots, 
a_m$ are nonzero integers. Then,
there exists $i \leq k\leq m$,
such that
 $$ -\deg~\sW \geq  (a_0+\cdots +
a_kp^k)\left[(h(a_{k+1}))({h}(a_{k+2})-1)\cdots
({h}(a_m)-1)\right],$$ where $h(t) = \dim~H^0(\bP^n_k, \sO_{\bP^n_k}(t))$.
\end{lemma}
\begin{proof}\quad By hypothesis it follows that $1\leq i \leq m$. If $i = 
m$ 
then, by Lemma~\ref{l2}, $-\deg\sW \geq 
a_0+a_1p+\cdots +a_mp^m$. Hence we can assume that $1\leq i\leq m-1$.

Let $f: H^0(\sO(a_0+\cdots+a_{m-1}p^{m-1}))\tensor
F^{m*}(\sV_{a_m})
 \by{} \sV_{a_0+\cdots+a_{m}p^m}$ be the canonical map.
 As in the proof of Lemma~\ref{l2}, we get a commutative 
diagram
of
$G$-bundles 
$$\begin{array}{c}
0{}\\
\downarrow\\
 0   \rightarrow  \sV_{(\oplus_{i=0}^{m-1}a_ip^i)}\tensor F^{*m}\sV_{a_m}
\rightarrow
\displaystyle{H^0(\sO(\oplus_{i=0}^{m-1}a_ip^i))}\tensor
F^{*m}\sV_{a_m} \rightarrow \sO(\oplus_{i=0}^{m-1}a_ip^i)\tensor
F^{*m}\sV_{a_m}  \rightarrow 0 \\
 \downarrow \\
\sV_{a_0+\cdots + a_mp^m} \\
 \downarrow \\
0  \longrightarrow \sV_{(\oplus_{i=0}^{m-1}a_ip^i)}\tensor\sO(a_mp^m)
  \longrightarrow 
\coker(f)  \longrightarrow 
\oplus~\sO_X  \longrightarrow  0\hspace{3cm}\\
\downarrow \\
 0
\end{array}$$

We note that, if $\sV'$ denotes the homogeneous subbundle given as 
$\sV' =$ kernel of the canonical composite map 
$$\sV_{a_0+\cdots+a_mp^m}\by{}\coker(f) \by{} \oplus~\sO_X $$ 
and $\sW'= \sW\cap \sV'$, then  
$-\deg{\sW} \geq -\deg~\sW'$.
Since $i \leq m-1$, we have $ W(a_0+\cdots+a_{m-1}p^{m-1}) = 0$. Therefore 
 $$\sW'\subseteq \sV_{a_0+\cdots+a_{m-1}p^{m-1}}\tensor
H^0(\sO(a_m))^{(p^m)},$$
as $G$-subbundle. Hence, using  Lemma~\ref{r1}, we see that $\sW'$ has a 
$G$-stable filtration with  
$$\mbox{gr}~\sW' = \bigoplus_{i=0}^{a_m}\sA_{0i}\tensor 
F^{*m}(S^{a_m-i}(\sV_1)\tensor
\sO(i)),$$
where $\sA_{0i} \subseteq \sV_{a_0+\cdots +a_{m-1}p^{m-1}}$ is a
homogeneous
$G$-subbundle. Note that, by 
Remark~\ref{r2} and Lemma~\ref{l7}, we have  $\sA_{00} \neq 0$.
We have $-\deg~\sW  \geq -\deg~\sW'$, where
$$-\deg\sW'
 =  \sum_{i=0}^{a_m}-(\deg~\sA_{0i})|S^{a_m-i}(\sV_1)| -
\sum_{i=0}^{a_m}|\sA_{0i}|~\deg~~F^{m*}(S^{a_m-i}(\sV_1) \tensor \sO(i))$$
$$ =  \sum_{i=0}^{a_m}-(\deg~\sA_{0i})|S^{a_m-i}(\sV_1)| -
p^m\sum_{i=0}^{a_m}|\sA_{0i}||S^{a_m-i}(\sV_1)|\left[~\mu(S^{a_m-i}(\sV_1))+\mu(\sO(i))\right]$$
\begin{equation}\label{e9}-\deg\sW'= 
\sum_{i=0}^{a_m}-(\deg~\sA_{0i})|S^{a_m-i}(\sV_1)| +
\frac{p^m}{n}\sum_{i=0}^{a_m}|\sA_{0i}||S^{a_m-i}(\sV_1)|(a_m-i-ni).
\end{equation}
Note that, by induction on $m$  we have 
$-\deg~\sA_{0i}\geq 0$,  for $0\leq i \leq a_m$.
Let
$$\begin{array}{lcl}
(\star) & = &
\displaystyle{\sum_{i=0}^{a_m}|\sA_{0i}||S^{a_m-i}(\sV_1)|(a_m-i-ni)}
\end{array}$$
Now applying the identity $(a+1)|S^{a+1}(\sV_1)| = n(|S^{a}(\sV_1)|+
\cdots + |S^{0}(\sV_1)|)$, we have 
$$\begin{array}{lcl}
(\star) 
& =  & n|S^{a_m-1}(V_1)|(|\sA_{00}|-|\sA_{01}|)\\
& & + n|S^{a_m-2}(V_1)|(|\sA_{00}|-|\sA_{02}| + |\sA_{01}|-|\sA_{02}|)\\
& & \vdots\\
& & + n|S^1(V_1)|\left(|\sA_{00}|-|\sA_{0(a_m-1)}| + \cdots +
|\sA_{0(a_m-2)}|-|\sA_{0(a_m-1)}|\right) + \\
& & n|S^0(V_1)|\left(|\sA_{00}|-|\sA_{0(a_m)}| + \cdots +
|\sA_{0(a_m-2)}|-|\sA_{0(a_m)}| +
|\sA_{0(a_m-1)}|-|\sA_{0(a_m)}|\right)\end{array}$$
Note that each term $|\sA_{0i}|-|\sA_{0j}|$ (with $j >i$) is always 
non-negative, as by Lemma~\ref{r1}~(2), we have $\sA_{0j}\subseteq 
\sA_{0i}$.

\vspace{5pt}
\noindent{Case~(1)} If $ W(a_mp^m) = 0$, then $\sA_{0a_m} = 0$; now 
we choose $1\leq 
j \leq a_m$ such
that $\sA_{0j} = 0$ and $\sA_{0(j-1)} \neq 0$. Then 
$$\begin{array}{lcl}
(\star) & \geq &  n|S^{a_m-j}(V_1)|(|\sA_{00}|-|\sA_{0j}|+\cdots +
|\sA_{0(j-1)}|-|\sA_{0j}|)+\\
 & & \cdots + n|S^0(V_1)|(|\sA_{00}|-|\sA_{0a_m}| 
+ \cdots + |\sA_{0(a_m-1)}|-|\sA_{0a_m}|)\\
&  = &  n\left(|\sA_{00}| + \cdots +
|\sA_{0(j-1)}|\right)(|S^{a_m-j}(\sV_1)|+
\cdots |S^0(\sV_1)|)\geq n^2\cdot j\cdot h^0(\sO(a_m-j)).\end{array}$$
Therefore $$\frac{p^m}{n}(\star) \geq p^mnjh^0(\sO(a_m-j)) \geq
p^m(a_m+1)\geq a_0 + \cdots + a_mp^m.$$

\vspace{5pt}

\noindent{Case~(2)}. Suppose $ W(a_mp^m) \neq 0$, so that $\sA_{0i} 
\neq 0$, for all $i$. Choose
$i$ to be the largest integer such that 
 $|A_{00}| = \cdots = |A_{0i}|$.
\begin{enumerate}
\item[(a)] If $i = a_m$ then $(\star) = 0$, and so  Equation~\ref{e9} 
becomes  
$$-\deg~\sW \geq -(\deg~\sA_{00})h^0(\sO(a_m)).$$
\item[(b)] If $i = a_m-1$, then 
$$-\deg~\sW  \geq  -(\deg~\sA_{00})(h^0(\sO(a_m))-1) + a_mp^m.$$
\item[(c)] If $i< a_m-1$ then
 $$\frac{p^m}{n}(\star) \geq p^mh^0(\sO(a_m-(i+1)))(i+1) \geq
(a_m+1)p^m.$$
\end{enumerate}

Hence we conclude, from case~(1) and case~(2) that either 
$$-\deg~\sW \geq  a_0+\cdots + a_mp^m  ~~\mbox{or}~~ 
-\deg~\sW \geq -(\deg~\sA_{00})(h(a_m)-1) +\delta_m,$$
 where
$ \delta_m = 
\min\{-\deg~\sA_{00}, a_mp^m\}.$
Let 
$$\sA_{m-1} := \sA_{00} \subseteq \sV_{a_0+\cdots + a_{m-1}p^{m-1}}.$$
Note that 
$$\sW(a_0) \neq 0 \implies \sW'(a_0) \neq 0 
\implies \sA_{00} \neq 0, ~~{i.e.}.~~\sA_{m-1}(a_0)\neq 0.
$$
 Then by replacing $\sW$, $\sW'$ and $\sV_d$ by the $G$-homogeneous
bundles $\sA_{m-1}$, $\sA'_{m-1}$
and $ \sV_{a_0+\cdots + a_{m-1}p^{m-1}}$ respectively
we have 
$$\mbox{gr}~\sA_{m-1}' = \bigoplus_{i=0}^{a_{m-1}}\sA_{1i}\tensor
F^{*m-1}(S^{a_{m-1}-i}(\sV_1)\tensor
\sO(i)),$$
where $\sA_{1i} \subseteq \sV_{a_0+\cdots +a_{m-2}p^{m-2}}$, for each $i$,
is a $G$-homogeneous subbundle.
 Then either 
$$-\deg~\sA_{m-1} \geq a_0+a_1p+\cdots+a_{m-1}p^{m-1}$$ 
or 
$$-(\deg~\sA_{m-1}) \geq -(\deg~\sA_{m-2})(h(a_{m-1})-1) +\delta_{m-1}, $$
where $\delta_{m-1} = {\rm min}\{-\deg~\sA_{m-2}, a_{m-1}p^{m-1}\}$.

Now inductively define $\sA_{m-i} = \sA_{(i-1)0}$. Equivalently, we define 
$A_{j-1}$
as a subset of $V_{a_0+a_1p+\cdots +a_{j-1}p^{j-1}}$ such that 
$$A_{j-1}\tensor
F^{*j}S^{a_j}(V_1)\tensor\cdots \tensor
F^{*m}S^{a_m}(V_1) \hspace{10cm}$$
$$\hspace{2cm} = \left(V_{a_0+\cdots +
a_{j-1}p^{j-1}}\tensor
F^{*j}S^{a_j}(V_1)\tensor\cdots \tensor
F^{*m}S^{a_m}(V_1)\right)\cap W.$$
Let $\delta_j = {\rm
min}\{-\deg~\sA_{j-1}, a_jp^j\}$.
 
Let $k$ be the largest integer such that  
$-\deg~\sA_{k} \geq a_0+\cdots +a_kp^k$. By Lemma~\ref{l2}, we have 
$k\geq i$. 
Now, for every $ l\geq k$, we have
$$-\deg~\sA_{l} \geq -\deg~\sA_{l+1}(h(a_{l})-1)+\delta_l. $$
 This implies that 
$$-\deg~\sW \geq -\deg~\sA_{k}
\left[(h(a_{k+1})-1)(h(a_{k+2})-1)\cdots
(h(a_{m})-1)\right] $$
$$ +\delta_{k+1}\left[(h(a_{k+2})-1)\cdots
(h(a_{m})-1)\right]  + \delta_{k+2}\left[(h(a_{k+3})-1)\cdots
(h(a_{m})-1)\right]+\cdots +\delta_m. $$
Since
$$\delta_{k+1} \geq  \min\{a_0+\cdots+a_kp^k, a_{k+1}p^{k+1}\}\geq 
a_0+a_1p+\cdots +a_kp^k,$$
as  $a_{k+1} \neq 0$,
 we have 
$$-\deg~\sW \geq (a_0+\cdots a_kp^k)
\left[(h(a_{k+1}))(h(a_{k+2})-1)\cdots
(h(a_{m})-1)\right].$$
This proves the lemma.
\end{proof}

\begin{remark}\label{r7} We can generalise the statement of 
Lemma~\ref{l3} as follows: Let $\sW\subset 
\sV_d$ be a $G$-subbundle such that  $W(a_0)\neq 0$.
 Let $i$ be the integer such 
that  $ W(a_0+\cdots
+a_{i-1}p^{i-1}) \neq 0$
and $W(a_0+\cdots + a_ip^i) = 0$. Then,
there exists $i \leq k\leq m$,
such that
 $$ -\deg~\sW \geq  (a_0+\cdots +
a_kp^k)\left[|S^{a_{k+1}}(V_1)||S^{a_{k+2}}(V_1)|\cdots
|S^{a_m}(V_1)|\right].$$
This follows from the argument given, in the above proof, that 
always (independent of the fact that some $a_j= 0$ or $\neq 0$, 
we have 
$$-\deg~\sW \geq a_0+\cdots + a_mp^m~~~\mbox{or}~~~-\deg~\sW \geq 
(-\deg~\sA_{00})|S^{a_m}(V_1)|.$$ 
Now, 
by induction on $m$, the result follows.
\end{remark}

\begin{cor}\label{c2}If $n\geq d/p$, then, for any $G$-subbundle 
$\sW\subset \sV_d$, we have $-\deg~\sW > d $.
\end{cor}
\begin{proof}\quad By Lemma~\ref{l6}, we can assume that $W(a_0)\neq 0$. 
By Corollary~\ref{c3},  we can also assume that $m\geq 2$. 
Now, if $-\deg~\sW\ngeq a_0+\cdots + a_mp^m $, then, by 
Remark~\ref{r7}, 
there exists an integer $k$ with $1\leq k \leq m-1$
 such that 
$$-\deg~\sW\geq (a_0+\cdots +
a_kp^k)\left(|S^{a_{k+1}}(V_1)|\cdots
|S^{a_m}(V_1)|\right),$$ 
where $a_k \neq 0$ and $k\geq 1$,  which implies  
$$-\deg~\sW\geq (a_0+a_kp^k)n\geq (a_0+a_kp^k)(d/p) > d.$$ This proves the 
corollary.\end{proof}

\begin{remark}\label{r6} 
If along with the hypothesis of 
 Lemma~\ref{l3}, we have the additional
conditions, namely  $a_{k+1} = \ldots,a_{m-1} = 1$ and $p\leq n$, then 
it is easy to see that
$$-\deg~\sW \geq  (a_0+\cdots +
a_kp^k)h(a_{k+1})h(a_{k+2})\cdots h(a_{m}).$$\end{remark}

\begin{lemma}\label{l4}. Let $\sW \subset \sV_d$ be a $G$-subbundle such 
that $W(a_0)\neq 0$. If  
$p\leq n$ and $a_2,\ldots, a_m \geq 1$  
 then,  $-\deg~\sW \geq d$. \end{lemma}
\begin{proof}By Lemma~\ref{l6}, we can assume that $W(a_0)\neq 0$.
Therefore, by Lemma~\ref{l3}, if $-\deg~\sW \ngeq a_0+\cdots+a_mp^m$,
then there exists $1\leq k \leq m-1$ such that 
$$-\deg~\sW \geq  (a_0+\cdots + a_kp^k){h}(a_{k+1})
({h}(a_{k+2})-1)
 \cdots ({h}(a_{m})-1)$$ 
Moreover if $a_{k+1} = \ldots = a_m = 1$ then, by Remark~\ref{r6}, we have 
$$-\deg~\sW \geq (a_0+\cdots + a_kp^k){h}(a_{k+1})
{h}(a_{k+2})
 \cdots {h}(a_{m}).$$
In this case 
$$\begin{array}{lcl}
-\deg~\sW & \geq & 
(a_0+\cdots + a_kp^k)h(a_{k+1})\cdots h(a_{m})\\ 
& \geq & (a_0+\cdots + a_kp^k)(n+1)^{m-k}\\
& \geq & (a_0+\cdots +a_kp^k)(p+1)^{m-k},\end{array}$$
which implies that $-\deg~\sW \geq a_0+\cdots +a_mp^m$.

If $a_t \geq 2$ for some $k < t \leq m$. Then $h(a_t)-1 \geq n(n+3)$. 
Therefore 
$$\begin{array}{lcl}
-\deg~\sW & \geq & 
(a_0+\cdots + a_kp^k)
a_{k+1}\cdots {\hat a_t}\cdots a_{m}(n(n+3))n^{m-k-1}\\
& \geq & a_kp^m(a_m+4)\geq a_o +a_1p+\cdots + a_mp^m.\end{array}$$ 
This proves the lemma.\end{proof}

\begin{lemma}\label{l5}Let $\sW \subseteq \sV_d$ be a $G$-subbundle such
that $p\geq n$ and $W(a_0)\neq 0$. Let $1\leq i\leq m-1$ be a 
nonnegative  integer
such that  $W(a_0+\cdots + a_{i-1}p^{i-1}) \neq 0$ and $W(a_0+\cdots + 
a_ip^i) = 0$. Moreover  
$$a_{i+1},
\ldots a_{m-1}, a_m \in  \{p-n+1, \cdots , p-1\}.$$ Then 
$-\deg~\sW \geq d$.
\end{lemma}

\begin{proof} By Lemma~\ref{l6}, we can assume that $W(a_0)\neq 0$. 
Moreover, by Lemma~\ref{l4}, we can assume that $p
>n$. If $\deg~\sW \geq a_0+\cdots +a_mp^m$ then the lemma follows. Hence, 
by 
Lemma~\ref{l3}, there exists an  integer $k$ such that $i\leq k \leq m-1$ 
and 
$$-\deg~\sW \geq (a_0+\cdots +
a_kp^k){h}(a_{k+1})(h({a_{k+2}})-1)
\cdots (h(a_{m})-1).$$
  Note that, by hypothesis,
 $a_m \geq p-n+1$,  which 
implies that $h(a_m) \geq
(p+1)a_m$.
Therefore 
$$-\deg~\sW \geq (a_kp^k){h}(a_{k+1})({ 
h}(a_{k+2})-1)\cdots (h(a_{m})-1).$$
\begin{enumerate}
\item If $1\leq k\leq m-2$ then 
$$\begin{array}{lcl}
-\deg~\sW & \geq & (a_kp^k)
{h}(a_{k+1})({h}({a_{k+2}})-1)
\cdots (h(a_m)-1) \\
& \geq &
\displaystyle{a_kp^k{\binom{p-n+1+n}{n}}
{\binom{p-n+1+n-1}{n-1}}^{m-2-k}(pa_m)} \\
& \geq &
\displaystyle{a_kp^k\frac{p(p+1)}{2}
\left(\frac{p(p-1))}{2}\right)^{m-2-k}
(pa_m)}  \\
 & \geq & a_mp^m\displaystyle{\frac{(p+1)}{2}}~~\mbox{as $p> n \geq 3$}\\
&  \geq & a_0+\cdots + a_mp^m.
\end{array}$$
\item If $k = m-1 $ then 
$$\begin{array}{lcl}
-\deg~\sW & \geq & 
(a_0+\cdots + a_{m-1}p^{m-1}){h}(a_m) \\
{} & \geq & (a_0+\cdots + a_{m-1}p^{m-1})((p+1)a_m)\\
{} & \geq & a_0+\cdots + a_mp^m = d.\end{array} $$
\end{enumerate}
This proves the lemma. \end{proof}

\begin{remark}\label{r8} Let  $\sV_d$ be a bundle on $\bP^n_k$  
with  $n\geq 3$ and let $\sW\subseteq \sV$ 
be a $G$-subbundle. Then we have proved 
\begin{enumerate}
\item if $W(a_0) = 0$, then, 
by Lemma~\ref{l6}, 
$$-\deg~\sW \geq  (|C_{0}(\sW)|+ \cdots + |C_{m-1}(\sW)|)
+ n|C_{m}(\sW)|,$$
where $C_i(\sW)$ is defined as in Remark~\ref{r3}, 
\item if $W(a_0)\neq 0$ and satisfies the hypothesis of 
Corollary~\ref{c3}, Corollary~\ref{c2}, Lemma~\ref{l4} or Lemma~\ref{l5}, 
then $$-\deg~\sW \geq a_0+a_1p+\cdots +a_mp^m.$$
 \end{enumerate}
In other words if $\sW\subseteq \sV_d$ is homogeneous $G$-subbundle 
satisfying the hypothesis of 
Corollary~\ref{c3}, Corollary~\ref{c2}, Lemma~\ref{l4} or Lemma~\ref{l5}, 
then, either 
\begin{enumerate}
\item $-\deg~\sW \geq a_0+a_1p+\cdots +a_mp^m$ or
\item $-\deg~\sW \geq  (|C_{0}(\sW)|+ \cdots + |C_{m-1}(\sW)|)
+ n|C_{m}(\sW)|.$
\end{enumerate}
\end{remark}

\section{Main results}

\noindent{\underline {Proof of Theorem~\ref{c1}}}.\quad 
If $\bP^n_k = \bP^2_k$ then stability of the bundle $\sV_d$  follows by 
Proposition~\ref{p1}. Hence the statement~(1).

 Therefore we can assume 
that $n\geq 3$. 
Let $\sW\subseteq \sV_d$ be a homogenous subbundle. 
Let $d = (a_o+a_1p+\cdots+a_mp^m)p^{i_0}$ be the integer satisfying the 
hypotheses of the theorem. We write $a_j$ as $b_{i_0+j}$, in particular we 
we write $d = (b_{i_0}+\cdots +b_{i_0+m}p^{m})p^{i_0}$, where $b_{i_0}$ 
and $b_{i_0+m}$ are positive integers.
Consider the
following
diagram of $G$-bundles:
$$\begin{array}{ccccccccc}
 & &  0 & {} & 0 & {} &  0 & &  \\
& &  \downarrow{} & {} & \downarrow{} & {} &  \downarrow{} & & \\
0 & \rightarrow & F^{*i_0}(\sV_{d'})
& \rightarrow{}  &
H^0(\sO(d'))^{(p^{i_0})}\tensor\sO_{{\mathbf
P}^n_k} & \rightarrow
&  \sO(d'p^{i_0}) & \rightarrow & 0 \\
 & &  \downarrow{f} & {}&   \downarrow{} & {}&   \downarrow{} & &  \\
 0 & \rightarrow & \sV_{d'p^{i_0}} &
 \rightarrow &
H^0(\sO(d'p^{i_0}))\tensor\sO_{{\mathbf
P}^n_k} & \rightarrow
 & \sO(d'p^{i_0}) & \rightarrow & 0 \\
& &  \downarrow{} & {} & \downarrow{} & {} & \downarrow{} & &  \\
& & \coker~(f) & {} & \oplus~\sO_{{\mathbf
P}^n_k}  & & {} 0 & & \\
 & &  \downarrow{} &  {} & \downarrow{}&  & {} & & {} \\
& &  0 & {} & 0 & {}  & {} & &  \\
\end{array}$$
where $d' = a_{0}+\cdots + a_{m}p^{m} = b_{i_0}+\cdots +b_{i_0+m}p^{m}$ 
and $H^0(\sO(a))$
denotes
$H^0({\bP^n}_k, \sO_{{\mathbf
P}^n_k}(a))$.
Therefore, for
any $G$-subbundle $\sW\subseteq \sV_d $,
we have
$$-\deg~\sW \geq -\deg~(\sW\cap
F^{*i_0}(\sV_{b_{i_0}+\cdots +b_{i_0+m}p^{m}})).$$
Then, by Lemma~\ref{r1},
we have $\sW\cap
F^{*i_0}(\sV_{b_{i_0}+ \cdots + b_{i_0+m}p^m}) = F^{*i_0}(\sW_1)$, 
for some $G$-subbundle
$\sW_1 \subseteq  \sV_{b_{i_0}+ \cdots + b_{i_0+_m}p^m}$ and
corresponding $P$-submodule $W_1\subseteq V_{b_{i_0}+ \cdots +
b_{i_0+m}p^m}$.
 
Now, if for some
$$j_0+j_1p+\cdots + j_{i_0-1}p^{i_0-1}+j_{i_0}p^{i_0}+ \cdots +
j_{i_0+m}p^{i_0+m} < b_{i_0}
p^{i_0}+ \cdots + b_{i_0+m}p^{i_0+m}$$
 we have
$$ W(j_0+j_1p+\cdots + j_{i_0-1}p^{i_0-1}+
j_{i_0}p^{i_0}+ \cdots + j_{i_0+m}p^{i_0+m}) \neq 0$$
then, by part~(1) of Remark~\ref{r2},
$$ W(j_{i_0}p^{i_0}+ \cdots + j_{i_0+m}p^{i_0+m}) \neq
0,$$ and hence by part~(2) of Remark~\ref{r2},
$$W_1(j_{i_0}+
\cdots + j_{i_0+m}p^{m}) \neq 0. $$
Therefore, if
$$\begin{array}{lcl}
C_{i_0}(\sW) & = & \{(j_0, \ldots, j_{i_0},
b_{i_0+1}, \ldots,
b_{i_0+m})\mid 0\leq j_{i_0} < b_{i_0},~~\mbox{and}~~ W_{j_0, \ldots,
j_{i_0},
b_{i_0+1}, \ldots, b_{i_0+m}} \neq 0\}\\
{} & \vdots & {} \\
C_{i_0+m}(\sW) & = & \{\{(j_0, \ldots, j_{i_0-1},j_{i_0}, \ldots,
j_{i_0+m})\mid 0\leq j_{i_0+m} < b_{i_0+m},~~\mbox{and}~~
W_{j_0, \ldots, j_{i_0+m}} \neq 0 \}\end{array}$$
 Then, $C_k(\sW) = \phi $, for $k < i_0$, and therefore 
$$\mbox{gr}~\sW =\bigoplus_{(j_0,\ldots, j_{i_0+m})\in
C_{i_0}(\sW)\cup\cdots
\cup
C_{i_0+m}(\sW)}\sW_{j_0,\cdots,
j_{i_0+m}}$$
and by Remark~\ref{r3},
$$-\mu(\sV_d)|W| \leq  (|C_{i_0}(\sW)|+ \cdots + |C_{i_0+m-1}(\sW)|)
+ n|C_{i_0+m}(\sW)|.$$
 On the other hand
$$\begin{array}{lcl}
C_{0}(\sW_1) & = & \{ (j_{i_0}, b_{i_0+1}, \ldots,
b_{i_0+m})\mid 0\leq j_{i_0} < b_{i_0},~~\mbox{and}~~(W_{1})_{j_{i_0},
b_{i_0+1}, \ldots, b_{i_0+m}}\neq 0\}\\
 & \vdots & {} \\
C_{t}(\sW_1) & = & \{(j_{i_0}, \ldots,  j_{i_0+t}, b_{i_0+{(t+1)}},
\ldots, b_{i_0+m})\mid 0\leq
j_{i_0+t} < b_{i_0+t},\\
 & & ~~\mbox{and}~~
(W_{1})_{j_{i_0}, \ldots, j_{i_0+t}, b_{i_0+{(t+1)}}, \ldots, b_{i_0+m}}
\neq
0\}\\
& \vdots & \\
C_{m}(\sW_1) & = & \{(j_{i_0}, \ldots,  j_{i_0+m})\mid 0\leq
j_{i_0+m} < b_{i_0+m},~~\mbox{and}~~
(W_{1})_{j_{i_0}, \ldots, j_{i_0+m}} \neq 0\}\end{array}$$
then
$$|C_{i_0}(\sW)| \leq p^{i_0}|C_{0}(\sW_1)|, \ldots,
|C_{i_0+m}(\sW)| \leq p^{i_0}|C_{m}(\sW_1)|.$$
 
Since $\sW_1\subseteq \sV_{b_{i_0}+ \cdots + b_{i_0+m}p^m}$  is 
a homogeneous $G$-subbundle 
satisfying the hypothesis of 
Corollary~\ref{c3}, Corollary~\ref{c2}, Lemma~\ref{l4} or Lemma~\ref{l5}, 
by Remark~\ref{r8}, we have either
$$\begin{array}{lcl}
-\deg~\sW_1 &  \geq & b_{i_0}+ \cdots + b_{i_0+m}p^m
\hspace{2cm}~~\mbox{or}~~\\
-\deg~\sW_1 & > & n(|C_{m}(\sW_1)|) + (|C_{0}(\sW_1)|+ \cdots +
|C_{m-1}(\sW_1)|).\end{array}$$

Since  $-\deg~\sW \geq p^{i_0}(-\deg~\sW_1)$, the above inequalities imply
that
$$\begin{array}{lcl}
-\deg~\sW & \geq & b_{i_0}p^{i_0}+ \cdots +
b_{i_0+m}p^{i_0+m} ~~\mbox{or}~~\\
-\deg~\sW & > &  np^{i_0}|C_{m}(\sW_1)| +
(p^{i_0}|C_{0}(\sW_1)|+
\cdots + p^{i_0}|C_{m-1}(\sW_1)|)\\ 
& \geq &  n|C_{i_0+m}(\sW)| + (|C_{i_0}(\sW)|+ \cdots + |C_{i_m}(\sW)|) \\
 & \geq & -\mu(\sV_d)|W|,\end{array}$$
where the last line follows from Remark~\ref{r3}.
Hence, in both the cases,  if $\sW \subset \sV$ then $\mu(\sW) < 
\mu(\sV_d)$. 

Now, due to the uniqueness property of the Harder-Narasimhan filtration,  
the destabilizing subbundle $\sW$ of
$\sV_d$
is a homogeneous $G$-subbundle such that $\mu(\sW) > \mu(\sV_d)$,
which
contradicts the result above.
In particular $\sV_d$ is semistable.
Now suppose $\sV_d$ is not stable 
then it has a subbundle $\sV' \subset \sV_d$ such that $\mu(\sV') =
\mu(V_d)$. Now Socle$(\sV')$ is the unique polystable subbundle
of same slope, containing
$\sV'$. In particular $\mbox{Socle}~(\sV')$ is homogeneous $G$-subbundle
of
same slope as $\sV_d$. Hence $\sV_d = \mbox{Socle}~(\sV')$ is polystable.

Now since $H^0(\bP^n_k,\sEnd(\sV_d))=k$, by a simple calculation we 
conclude that, $\sV_d$
must be stable.

This proves the theorem. $\Box$
 
\vspace{5pt}

\noindent{\it Proof of Proposition~\ref{p2}}\quad If $\sV_d^*$ is 
semistable then 
$$\mu_{max}(\sV_d^*) = \mu(\sV_d^*) = \frac{d}{{\binom{d+n}{d}}-1}.$$
Hence we assume that $\sV_d^*$ is not semistable. Let 
$$0 = \sU_0 \subset \sU_1\subset \cdots \subset \sU_l = \sV_d^*$$ 
be the Harder-Narasimhan filtration of $\sV_d^*$. 
Then, by definition,  $\mu_{max}(\sV_d^*) = \mu(\sU_1)$.
Note that  $\sU_1$ is a $G$-subbundle of $\sV_d^*$,  and therefore  there 
exists  corresponding $P$-submodule, say, $U_1$ of $V_d$.
Consider the short exact sequence of $G$-bundles 
$$0\by{} \sU_1 \by{} \sV_d^* \by{} \sE \by{} 0,$$
Taking dual of this, we get a short exact sequence of $G$-bundles
$$ 0\by{} \sE^* \by{} \sV_d^{**} = \sV_d  \by{} \sU_1^* \by{} 0.$$ 
Now let us denote the $G$-subbundle 
$\sE^*$ by $\sW$ and let 
 $W$ be the corresponding $P$-submodule of $V_d$. 
Now, the first inequality of the proposition follows from 
hypothesis that 
$\mu(\sV_d^*) < \mu(\sU_1) = \mu_{max}(\sV_d^*) $.
Moreover 
$$\mu(\sU_1) > \mu(\sV_d^*)\implies \mu(\sE) < \mu(\sV_d^*) 
\implies  \mu(\sW) > \mu(\sV_d).$$ 
Now, by Lemma~\ref{l6}, we have 
$W(a_0) \neq 0$ and therefore, by Remark~\ref{r7}, we have  
$\deg~\sW < 0$. This gives 
 $$\deg~\sU_1  =  \deg~\sV_d^* -\deg~\sE =  \deg~\sV_d^* + \deg~\sW 
< d.$$

 By Lemma~\ref{l2},
$W(a_0+\cdots + a_{m-1}p^{m-1}) = 0$, as $\mu(\sW) > \mu(\sV_d)$.
Hence, by Remark~\ref{r2}, for every 
$0\leq i_m \leq a_m-1$, we have 
$$W_{a_0, \ldots, a_{m-1}, i_m} = W(a_0+\cdots+ a_{m-1}p^{m-1}+i_mp^m) = 
0.$$ 
Now
$$\begin{array}{lcl}
|U_1| & = & |V_d| - |W|\\
 & = & \displaystyle{\sum_{\{(i_0,\ldots, i_m)\mid 0\leq i_j\leq p-1, 
~~i_0+\cdots + i_mp^m < d\}} 
|V_{i_0, \ldots, i_m}| - |W_{i_0,\ldots,i_m}|}
\end{array}$$
As $|V_{i_0, \ldots, i_m}| - |W_{i_0,\ldots,i_m}| \geq 0$, for every tuple
$(i_0, \ldots, i_m)$, this implies
$$\begin{array}{lcl}
|U_1| & \geq & \displaystyle{\sum_{\{i_m \mid 0\leq i_m \leq 
a_m-1\}}|V_{a_0,\ldots,a_{m-1},i_m}| - |W_{a_0,\ldots,a_{m-1},i_m}|}\\
& \geq & \displaystyle{\sum_{0\leq i_m \leq a_m-1}|V_{a_0, \ldots , 
a_{m-1}, i_m}|}
= |S^{a_mp^m}(V_1)| + \cdots + |S^{p^m}(V_1)| \geq 
|S^{{\lceil{d/2}\rceil}}(V_1)|.
\end{array}
$$
In particular 
$\mu(\sU_1)\leq d/|S^{\lceil{d/2}\rceil}(V_1)|$. 
 Hence the proposition. $\Box$

\vspace{7pt}

\noindent{\it Proof of Corollary~\ref{c4}}\quad By the proof of 
Theorem~\ref{t1} of Langer; if 
$$d > \frac{r-1}{r}\Delta(E)H^{n-2}+\frac{1}{r(r-1)H^n}$$ and $m$ the 
least 
integer such that the quotients of the Harder-Narasimhan filtration of the 
restriction 
of 
$E$ to a very general divisor in $dH$ are strongly semistable, then for 
general hypersurface $D$ in $|dH|$, we have 
$$ \frac{d}{\max\{\frac{r^2-1}{4}, 1\}} \leq \mu_i(F^{m*}E|_D)-
\mu_{i+1}(F^{m*}E|_D) \leq H^n\cdot\mu_{max}(\sV_d^*|_D).$$
By Proposition~\ref{p2}, this implies 
$$ \frac{d}{\max\{\frac{r^2-1}{4}, 1\}} \leq \mu_i(F^{m*}E|_D)-
\mu_{i+1}(F^{m*}E|_D) \leq H^n\cdot 
\frac{d^2}{\binom{ {\lceil{d/2}\rceil}+n-1}{{\lceil{d/2}\rceil}}}.$$
Moreover, for $n =2$, by Proposition~\ref{p1},
we have
$$ \frac{d}{\max\{\frac{r^2-1}{4}, 1\}} \leq \mu_i(F^{m*}E|_D)-
\mu_{i+1}(F^{m*}E|_D) \leq H^n\cdot 
\frac{d^2}{\binom{d+n}{d}-1}.$$  
Now, for $d$ such that 
\begin{enumerate}
\item for $n = 2$, the inequality 
$$\frac{\binom{d+n}{d}-1}{d} > H^n\cdot \max\{\frac{r^2-1}{4}, 
1\}+1$$
holds and
\item for $n\geq 3$, 
the inequality  
$$\frac{\binom{ {\lceil{d/2}\rceil}+n-1}{{\lceil{d/2}\rceil}}}{d}
> H^n\cdot \max\{\frac{r^2-1}{4}, 1\}+1$$
\end{enumerate}
holds, we have a contradiction. In particular, for $d$ satisfying 
the 
hypothesis of the corollary, $E_D$ is 
strongly semistable, for a very 
general $D\in |dH|$.$\Box$

 \end{document}